\documentclass[reqno]{amsart}%
\usepackage{amssymb,mathtools,lineno}

\usepackage{ifpdf}
\ifpdf
\usepackage[hyperindex]{hyperref}
\else
\expandafter\ifx\csname dvipdfm\endcsname\relax
\usepackage[hypertex,hyperindex]{hyperref}
\else
\usepackage[dvipdfm,hyperindex]{hyperref}
\fi
\fi

\numberwithin{equation}{section}

\allowdisplaybreaks[4]

\theoremstyle{plain}
\newtheorem{thm}{Theorem}[section]
\newtheorem{cor}{Corollary}[section]
\newtheorem{lem}{Lemma}[section]
\theoremstyle{remark}
\newtheorem{rem}{Remark}[section]

\DeclareMathOperator{\td}{d\!}
\DeclareMathOperator{\ti}{i}
\DeclareMathOperator{\te}{e}
\DeclareMathOperator{\Arg}{Arg}
\DeclareMathOperator{\Ln}{Ln}
\DeclareMathOperator{\Arctan}{Arctan}
\DeclareMathOperator{\arctanh}{arctanh}
\DeclareMathOperator{\bell}{B}

\begin{document}

\title[A power series expansion of the Wilf function]
{A power series expansion of the Wilf function}

\author[F. Qi]{Feng Qi}
\address{School of Mathematics and Physics, Hulunbuir University, Hulunbuir 021008, Inner Mongolia, China; 
17709 Sabal Court, Dallas, TX 75252-8024, USA}
\email{\href{mailto:qifeng618@gmail.com}{qifeng618@gmail.com}}
\urladdr{\url{https://orcid.org/0000-0001-6239-2968}}

\begin{abstract}
In this work, the author employs the Fa\`a di Bruno formula, identities for the partial Bell polynomials, two combinatorial identities, and the (logarithmically) complete monotonicity of generating functions for several integer sequences, together with the Wronski theorem, to investigate a collection of analytic and combinatorial structures. The study establishes Taylor series expansions for various functions involving the inverse (hyperbolic) tangent function and derives the Maclaurin expansion of the Wilf function, a composite of the inverse tangent, square root, and exponential functions. The coefficients in this expansion are expressed in terms of Stirling numbers of the second kind, and their generating functions, limits, positivity, monotonicity, and logarithmic convexity are analyzed. The paper further presents closed-form formulas for special values of the Gauss hypergeometric function and for certain partial Bell polynomials, along with several infinite series representations of the circular constant and related sequences. An asymptotic rational approximation to the circular constant is recovered, and connections among several integer sequences are established via determinants.
\end{abstract}



\thanks{This paper was typeset using \AmS-\LaTeX}

\maketitle

\section{Introduction and motivation}

Throughout this paper, we use the following notations:
\begin{gather*}
\mathbb{N}=\{1,2,\dotsc\}, \quad \mathbb{N}_0=\{0,1,2,\dotsc\},\quad \mathbb{N}_-=\{-1,-2,\dotsc\}, \quad \mathbb{Z}=\{0,\pm1,\pm2,\dotsc\}, \\
\mathbb{R}=(-\infty,\infty), \quad \mathbb{C}=\bigl\{x+y\ti: x,y\in\mathbb{R}, \ti=\sqrt{-1}\,\bigr\}.
\end{gather*}
\par
Let
\begin{equation}\label{Wilf-F-arctan}
W(z)=\frac{\arctan\sqrt{2\te^{-z}-1}\,}{\sqrt{2\te^{-z}-1}\,}.
\end{equation}
In 2009, Herbert S. Wilf (1931--2012, University of Pennsylvania, USA) proposed the following problem:
\begin{quote}
If
\begin{equation}\label{Wilf-a(n)-dfn}
W(z)=\sum_{n=0}^{\infty}a_nz^n,
\end{equation}
find the first term of the asymptotic behavior of the $a_n$'s.
\end{quote}
\par
In the proceedings paper~\cite{Ward-2010-Wilf}, Mark Daniel Ward (Purdue University, USA) considered Wilf's problem and obtained the following conclusions:
\begin{enumerate}
\item
The function $\sqrt{2\te^{-z}-1}\,$ is analytic on $\mathbb{C}\setminus\bigcup_{j\in\mathbb{Z}}\mathcal{L}_j$, where
\begin{equation*}
\mathcal{L}_j=\{z=x+2j\pi\ti:x\in\mathbb{R},x\ge\ln2\}.
\end{equation*}
\item
The Wilf function $W(z)$ defined in~\eqref{Wilf-F-arctan} is analytic on $\mathbb{C}\setminus\bigcup_{j\in\mathbb{Z}}\mathcal{L}_j$. Moreover, an appropriate choice for the branch of the $\arctan$ function allows $W(z)$ to have an analytic continuation to $\bigcup_{j\in2\mathbb{Z}}\mathcal{L}_j$; however, such an analytic continuation will leave $W(z)$ discontinuous on $\bigcup_{j\in2\mathbb{Z}+1}\mathcal{L}_j$.
\item\label{Ward-item3}
The coefficients $a_n$ for $n\in\mathbb{N}_0$ in the Maclaurin power series expansion~\eqref{Wilf-a(n)-dfn} satisfy $a_n=\frac{h(n)}{R^n}$, where $R=|2\pi\ti+\ln2|$ and $h(n)$ is a subexponential factor, i.e., $\limsup_{n\to\infty}|h(n)|^{1/n}=1$.
\item
For each positive integer $N$,
\begin{align*}
a_n&=-2\pi\sum_{j=0}^{N-1}\bar{d}_j\frac{\cos\bigl[\bigl(n-j+\frac{1}{2}\bigr)\arctan\frac{2\pi}{\ln2}\bigr]} {\Gamma\bigl(\frac{1}{2}-j\bigr)R^{n-j+1/2}n^{j+1/2}} \Biggl[1+\sum_{k=1}^{N-j-1}\frac{\bar{e}_k\bigl(\frac{1}{2}-j\bigr)}{n^k}\Biggr]\\
&\quad +O\biggl(\frac{1}{R^nn^{N+1/2}}\biggr),
\end{align*}
where
\begin{gather*}
\bar{d}_j=\bigl[z^{j-1/2}\bigr]\frac1{(\te^z-1)^{1/2}},\quad
\lambda_{k,\ell}=\bigl[v^kt^\ell\bigr]\frac{\te^t}{(1+vt)^{1+1/v}},\\
\bar{e}_k(\alpha)=\sum_{\ell=k}^{2k}(-1)^\ell\lambda_{k,\ell}\prod_{q=1}^{\ell}(\alpha-q),
\end{gather*}
and $\Gamma$ denotes the classical Euler gamma function which can be defined~\cite[Chapter~3]{Temme-96-book} by
\begin{equation*}
\Gamma(z)=\lim_{n\to\infty}\frac{n!n^z}{\prod_{k=0}^n(z+k)}, \quad z\in\mathbb{C}\setminus\{0,-1,-2,\dotsc\}.
\end{equation*}
\end{enumerate}
\par
Tian-Xiao He (Illinois Wesleyan University, USA) reviewed the paper~\cite{Ward-2010-Wilf} in \href{https://mathscinet.ams.org/mathscinet-getitem?mr=2735366}{MR2735366} as follows:
\begin{quote}
The coefficient $a_n$ can be written as $a_n=b_n\pi-c_n$, where $b_n$ and $c_n$ are nonnegative rational numbers. In fact, $\lim_{n\to\infty}\frac{a_n}{b_n}=0$, and the rational numbers of the form $\frac{c_n}{b_n}$ provide approximations to $\pi$. A complete expansion of the coefficients $a_n$ is found by the author. It is probably the best that can be done, given the oscillatory nature of the terms.
\end{quote}
Herbert S. Wilf commented on the paper~\cite{Ward-2010-Wilf} on 13 December 2010 as follows:
\begin{quote}
Mark Ward has found a complete expansion of these coefficients. It's not quite an asymptotic series in the usual sense, but it is probably the best that can be done, given the oscillatory nature of the terms.
\end{quote}
\par
In this paper, we will establish a closed-form formula involving the Stirling numbers of the second kind $S(n,k)$ for the coefficients $a_n$ in~\eqref{Wilf-a(n)-dfn}. Using this closed-form formula, we will express $a_n$ in the form $b_n\pi-c_n$ in terms of closed-form formulas for $b_n$ and $c_n$.
\par
To smoothly attain our aims above, we will prepare some knowledge, tools, and lemmas in Section~\ref{preparation-sec}.
\par
For establishing a closed-form formula for the coefficients $a_n$, we not only need to derive the Taylor or the Maclaurin power series expansions of the functions
\begin{equation}\label{four-arctangent}
\arctan z,\quad \sqrt{2\te^{-z}-1}\,,\quad \frac{\arctan\sqrt{z}\,}{\sqrt{z}\,},\quad \frac{\arctan z}{z},
\end{equation}
but also need to derive a closed-form formula for the partial Bell polynomials
\begin{equation}\label{Bell-Poly-Stirling-invol}
\bell_{n,k}\Biggl(1, 0, 1, 3, 16, 105, \dotsc, (-1)^{n-k}\sum_{\ell=0}^{n-k+1} (-1)^{\ell-1}S(n-k+1,\ell)(2\ell-3)!!\Biggr)
\end{equation}
for $n\ge k\in\mathbb{N}_0$, where $S(n,k)$ denotes the Stirling numbers of the second kind and the double factorial of negative odd integers $-(2k+1)$ for $k\in\mathbb{N}_0$ is defined by
\begin{equation}\label{double-factorial-neg-odd-integer}
[-(2k+1)]!!=\frac{(-1)^k}{(2k-1)!!}=(-1)^k\frac{(2k)!!}{(2k)!}.
\end{equation}
We will carry out these computations in Section~\ref{arctan-taylor-maclaurin-sec}.
\par
In Section~\ref{sec-Maclaurin-Wilf-Ser}, by means of the Taylor power series expansion derived in Section~\ref{arctan-taylor-maclaurin-sec} of the third function in~\eqref{four-arctangent}, we will present two forms for the Maclaurin power series expansion of the Wilf function $W(z)$ on the open disc $|z|<\ln2$: the first form is expressed in terms of the Gauss hypergeometric function
\begin{equation}\label{2F1nnn-z-1}
{}_2F_1\biggl(n+\frac{1}{2},n+1;n+\frac{3}{2};-1\biggr), \quad n\in\mathbb{N}_0;
\end{equation}
the second one is represented in terms of a closed-form expression for the coefficients $a_n$. Consequently, we deduce a closed-form formula for the Gauss hypergeometric function in~\eqref{2F1nnn-z-1}.
\par
In Section~\ref{ser-arctanh-sec}, in view of related results obtained in Sections~\ref{arctan-taylor-maclaurin-sec} and~\ref{sec-Maclaurin-Wilf-Ser}, in light of a relation between $\arctan z$ and $\arctanh z$, we will establish the Maclaurin power series expansions of four functions
\begin{equation*}
\arctanh z,\quad \frac{\arctanh\sqrt{z}\,}{\sqrt{z}\,},\quad \frac{\arctanh z}{z},\quad \frac{\arctanh\sqrt{1-2\te^{-z}}\,}{\sqrt{1-2\te^{-z}}\,}.
\end{equation*}
\par
In Section~\ref{coefficients-analysis-sec}, we will go back to the Maclaurin power series expansion
\begin{equation}\label{Wilf-F-Ser-Split2Part}
W(z)=\sum_{n=0}^{\infty}a_nz^n
=\sum_{n=0}^{\infty}(b_n\pi-c_n)z^n, \quad |z|<\ln2
\end{equation}
and analyze many properties of the sequences $a_n$, $b_n$, and $c_n$.

\section{Preliminaries and lemmas}\label{preparation-sec}

In the theory of complex functions, the inverse tangent function $\Arctan z$, viewed as a multivalued analytic function, can be represented by
\begin{equation*}
\Arctan z=\frac1{2\ti}\Ln\frac{1+\ti z}{1-\ti z}=\frac1{2\ti}\Ln\frac{\ti-z}{\ti+z}
\end{equation*}
on the punctured complex plane $\mathbb{C}\setminus\{\pm\ti\}$. Its principal branch $\arctan z$, normalized by $\arctan 0=0$, is taken as
\begin{equation*}
\arctan z=\frac{\ti}2[\ln(1-\ti z)-\ln(1+\ti z)]
\end{equation*}
on the cut complex plane
\begin{equation*}
\mathcal{A}=\mathbb{C}\setminus\{(-\ti\infty,-\ti],[\ti,\ti\infty)\},
\end{equation*}
where $\ln z=\ln|z|+\ti\arg z$ for $z\in\mathbb{C}\setminus(-\infty,0]$ stands for the principal branch of the logarithm $\Ln z$ and $\arg z\in(-\pi,\pi]$ represents the principal value of the argument $\Arg z$.
On the cut complex plane $\mathcal{A}$, the range of the principal branch $\arctan x$ for real numbers $x\in\mathbb{R}$ is determined as $\bigl(-\frac\pi2,\frac\pi2\bigr)$.
The principal branch $\arctan z$ has a unique zero $z=0$ on $\mathcal{A}$.
\par
We use the principal branch of $f(z)=\sqrt{z}\,$ on $\mathbb{C}\setminus(-\infty,0]$. In polar coordinates, we can write $z=r\te^{\ti\theta}$ for $r>0$ and $\theta\in(-\pi,\pi]$. Then $f(z)=\sqrt{r}\,\te^{\ti(\theta+2k\pi)/2}$ for $k=0,1$. On the cut complex plane $\mathbb{C}\setminus(-\infty,0]$, the principal branch of $f(z)$ such that $f(1)=1$ is $\sqrt{r}\,\te^{\ti\theta/2}$ for $r>0$ and $\theta\in(-\pi,\pi)$. The range of the principal branch of $f(z)=\sqrt{z}\,$ is the open right half complex plane. This range is a subset of $\mathcal{A}$. Therefore, the principal branch of the third function in~\eqref{four-arctangent} is defined on $\mathbb{C}\setminus(-\infty,0]$.
\par
The function $2\te^{-z}-1$ is analytic on the complex plane $\mathbb{C}$ and its range is $\mathbb{C}\setminus\{-1\}$.
\par
It is obvious that the Wilf function $W(z)$ can be regarded as the composite of three complex functions
\begin{equation*}
h(z)=\begin{dcases}
\frac{\arctan z}{z}, & z\ne0;\\
1, & z=0,
\end{dcases}
\quad f(z)=\sqrt{z}\,,\quad g(z)=2\te^{-z}-1.
\end{equation*}
When considering the composite of three principal branches of these complex functions $h(z)$, $f(z)$, and $g(z)$, the Wilf function $W(z)$ has no definition on points such that
\begin{equation*}
2\te^{-z}-1=2\te^{-x}\cos y-1-(2\te^{-x}\sin y)\ti\in(-\infty,0),
\end{equation*}
which is equivalent to
\begin{equation*}
2\te^{-x}\cos y-1<0 \quad\text{and}\quad 2\te^{-x}\sin y=0.
\end{equation*}
These points constitute infinitely many parallel and horizontal rays
\begin{equation*}
\{z=x+y\ti\in\mathbb{C}:x>\ln2, y=2k\pi, k\in\mathbb{Z}\}
\end{equation*}
and infinitely many parallel, horizontal, and straight lines
\begin{equation*}
\{z=x+y\ti\in\mathbb{C}:x\in\mathbb{R}, y=(2k+1)\pi,k\in\mathbb{Z}\}
\end{equation*}
on the complex plane $\mathbb{C}$. Thus, the Wilf function $W(z)$ is defined on the region
\begin{align*}
\mathbb{C}&\setminus\{z=x+y\ti\in\mathbb{C}:x>\ln2, y=2k\pi, k\in\mathbb{Z}\}\\
&\setminus\{z=x+y\ti\in\mathbb{C}:x\in\mathbb{R}, y=(2k+1)\pi,k\in\mathbb{Z}\}.
\end{align*}
Consequently, on the open disc $|z|<\ln2$, the Wilf function $W(z)$ is analytic.
\par
By choosing an alternative branch of $\arctan z$, Ward proved in~\cite[Lemma~2.2]{Ward-2010-Wilf} that the maximum radius around the origin for which the Wilf function $W(z)$ has an analytic extension is
\begin{equation*}
|\ln2+2\pi\ti|=\sqrt{(\ln2)^2+4\pi^2}\,=6.32\dotsc>\ln2=0.69\dotsc.
\end{equation*}
In what follows, we will proceed by assuming $|z|<\ln2$, if any.
\par
For a complex constant $\alpha\in\mathbb{C}$, the falling and rising factorials of $\alpha$ are defined~\cite[p.~618, 26.1.1]{NIST-HB-2010} respectively by
\begin{equation*}
\langle\alpha\rangle_n=\prod_{k=0}^{n-1}(\alpha-k)
=
\begin{cases}
\alpha(\alpha-1)\dotsm(\alpha-n+1), & n\in\mathbb{N};\\
1, & n=0
\end{cases}
\end{equation*}
and
\begin{equation*}
(\alpha)_n=\prod_{k=0}^{n-1}(\alpha+k)
=
\begin{cases}
\alpha(\alpha+1)\dotsm(\alpha+n-1), & n\in\mathbb{N};\\
1, & n=0.
\end{cases}
\end{equation*}
The generalized binomial coefficient $\binom{z}{w}$ for $z,w\in\mathbb{C}$ was defined in~\cite[p.~783]{2nd-Bell-Polyn-factoria-vl.tex} by
\begin{equation*}
\binom{z}{w}=
\begin{dcases}
\frac{\Gamma(z+1)}{\Gamma(w+1)\Gamma(z-w+1)}, & z\not\in\mathbb{N}_-,\quad w,z-w\not\in\mathbb{N}_-;\\
0, & z\not\in\mathbb{N}_-,\quad w\in\mathbb{N}_- \text{ or } z-w\in\mathbb{N}_-;\\
\frac{\langle z\rangle_w}{w!},& z\in\mathbb{N}_-, \quad w\in\mathbb{N}_0;\\
\frac{\langle z\rangle_{z-w}}{(z-w)!}, & z,w\in\mathbb{N}_-, \quad z-w\in\mathbb{N}_0;\\
0, & z,w\in\mathbb{N}_-, \quad z-w\in\mathbb{N}_-;\\
\infty, & z\in\mathbb{N}_-, \quad w\not\in\mathbb{Z}.
\end{dcases}
\end{equation*}
In particular, we have $\binom{-1}{-1}=1$ and $\binom{z}{-1}=0$ for $z\in\mathbb{Z}\setminus\{-1\}$.
\par
For $\alpha,\beta\in\mathbb{C}$ and $\gamma\in\mathbb{C}\setminus\{0,-1,-2,\dotsc\}$, the Gauss hypergeometric function ${}_2F_1(\alpha,\beta;\gamma;z)$ is defined by the series
\begin{equation}\label{Gauss-HF-dfn}
{}_2F_1(\alpha,\beta;\gamma;z)=\sum_{n=0}^{\infty}\frac{(\alpha)_n(\beta)_n}{(\gamma)_n}\frac{z^n}{n!}.
\end{equation}
If $\alpha$ or $\beta$ is zero or a negative integer, the series in~\eqref{Gauss-HF-dfn} terminates after a finite number of terms, and the Gauss hypergeometric function ${}_2F_1(\alpha,\beta;\gamma;z)$ is a polynomial in $z$. When $\alpha,\beta\not\in\{0,-1,-2,\dotsc\}$, the infinite series in~\eqref{Gauss-HF-dfn} converges absolutely and uniformly on the open unit disc $|z|<1$. If $\Re(\alpha+\beta-\gamma)<0$, the series in~\eqref{Gauss-HF-dfn} converges absolutely on the unit circle $|z|=1$. If $0\le\Re(\alpha+\beta-\gamma)<1$, the series in~\eqref{Gauss-HF-dfn} conditionally converges at all points of the unit circle $|z|=1$ except $z=1$. If $\Re(\alpha+\beta-\gamma)\ge1$, the series in~\eqref{Gauss-HF-dfn} diverges over the circle $|z|=1$. There exists an analytic continuation of the Gauss hypergeometric function ${}_2F_1(\alpha,\beta;\gamma;z)$ to the region $\{z\in\mathbb{C}:|z|>1\}\setminus(1,\infty)$. The principal analytic continuation is defined on $\mathbb{C}\setminus[1,\infty)$. For details and more information, please refer to~\cite[pp.~64--66, Section~3.7]{Luke-Vol1-1969} and~\cite[Chapter~5]{Temme-96-book}.
\par
The Stirling numbers of the first kind $s(n,k)$ for $n,k\in\mathbb{N}_0$ with $n\ge k$ can be analytically generated~\cite[Theorem~3.14]{Mansour-Schork-B2016} by
\begin{equation}\label{Stirl-No-First-GF}
\biggl[\frac{\ln(1+x)}{x}\biggr]^n=\sum_{k=0}^\infty \frac{s(n+k,n)}{\binom{n+k}{n}}\frac{x^{k}}{k!}, \quad |x|<1.
\end{equation}
The Stirling numbers of the second kind $S(n,k)$ for $n,k\in\mathbb{N}_0$ with $n\ge k$ can be analytically generated~\cite[p.~51]{Comtet-Combinatorics-74} by
\begin{equation}\label{2stirl-gen-f}
\biggl(\frac{\te^x-1}{x}\biggr)^n=\sum_{k=0}^{\infty} \frac{S(n+k,n)}{\binom{n+k}{n}}\frac{x^{k}}{k!}
\end{equation}
and can be explicitly computed~\cite[p.~204, Theorem~A]{Comtet-Combinatorics-74} by
\begin{equation}\label{Stirling-Number-dfn}
S(n,k)=
\begin{dcases}
\frac{(-1)^{k}}{k!}\sum_{\ell=0}^k(-1)^{\ell}\binom{k}{\ell}\ell^{n}, & n>k\in\mathbb{N}_0;\\
1, & n=k\in\mathbb{N}_0.
\end{dcases}
\end{equation}
We recommend the papers~\cite{MIA-4666.tex, 2900-Qi.tex} for several new results of $S(n,k)$.
\par
The partial Bell polynomials $\bell_{n,k}(x_1,x_2,\dotsc,x_{n-k+1})$ for $n\ge k\in\mathbb{N}_0$ can be defined~\cite[p.~134, Theorem~A]{Comtet-Combinatorics-74} by
\begin{equation*}
\bell_{n,k}(x_1,x_2,\dotsc,x_{n-k+1})=\sum_{\substack{1\le i\le n-k+1\\ \ell_i\in\{0\}\cup\mathbb{N}\\ \sum_{i=1}^{n-k+1}i\ell_i=n\\
\sum_{i=1}^{n-k+1}\ell_i=k}}\frac{n!}{\prod_{i=1}^{n-k+1}\ell_i!} \prod_{i=1}^{n-k+1}\biggl(\frac{x_i}{i!}\biggr)^{\ell_i}.
\end{equation*}
The Fa\`a di Bruno formula can be described~\cite[p.~139, Theorem~C]{Comtet-Combinatorics-74} in terms of the partial Bell polynomials $\bell_{n,k}(x_1,x_2,\dotsc,x_{n-k+1})$ by
\begin{equation}\label{Bruno-Bell-Polynomial}
\frac{\operatorname{d}^n}{\td t^n}f\circ h(t)
=\sum_{k=0}^nf^{(k)}(h(t)) \bell_{n,k}\bigl(h'(t),h''(t),\dotsc,h^{(n-k+1)}(t)\bigr), \quad n\in\mathbb{N}_0,
\end{equation}
where $f$ and $h$ are $n$-times differentiable functions and $f\circ h$ denotes the composite of $f$ and $h$.
\par
The partial Bell polynomials $\bell_{n,k}(x_1,x_2,\dotsc,x_{n-k+1})$ satisfy the following identities and formulas:
\begin{align}\label{Bell(n-k)}
\bell_{n,k}\bigl(\alpha\beta x_1,\alpha \beta^2x_2,\dotsc,\alpha \beta^{n-k+1}x_{n-k+1}\bigr)
&=\alpha^k\beta^n\bell_{n,k}(x_1,x_2,\dotsc,x_{n-k+1}),\\
\label{Bell-stirling}
\bell_{n,k}(1,1,\dotsc,1)
&=S(n,k),\\
\label{Bell-x-1-0-eq}
\bell_{n,k}(\alpha,1,0,\dotsc,0)
&=\frac{(n-k)!}{2^{n-k}}\binom{n}{k}\binom{k}{n-k}\alpha^{2k-n},\\
\frac1{k!}\Biggl(\sum_{m=1}^{\infty} x_m\frac{t^m}{m!}\Biggr)^k
&=\sum_{n=k}^{\infty} \bell_{n,k}(x_1,x_2,\dotsc,x_{n-k+1})\frac{t^n}{n!},\label{113-final-formula}
\end{align}
and
\begin{multline}\label{Bell-Polyn-Half}
\bell_{n,k}\biggl(\biggl\langle\frac12\biggr\rangle_1, \biggl\langle\frac12\biggr\rangle_2,\dotsc, \biggl\langle\frac12\biggr\rangle_{n-k+1}\biggr)\\
=(-1)^{n+k}\frac{(2n-2k-1)!!}{2^n} \binom{2n-k-1}{k-1}
\end{multline}
for $n\ge k\in\mathbb{N}_0$, $\alpha\in\mathbb{C}\setminus\{0\}$, and $\beta\in\mathbb{C}$; see~\cite[p.~412]{Charalambides-book-2002},~\cite[pp.~133 and 135]{Comtet-Combinatorics-74},~\cite[Theorem~4.1]{Spec-Bell2Euler-S.tex}, and~\cite[p.~169, (3.6)]{CDM-68111.tex}, respectively. These identities and formulas can also be found in the review and survey article~\cite{Bell-value-elem-funct.tex}.

\begin{lem}\label{sum-central-binom-lem-ell}
For all $\ell,n\in\mathbb{N}_0$ with $n\ge\ell$, the combinatorial identity
\begin{equation}\label{sum-central-binom-ell-eq}
\sum_{k=\ell}^{n}\binom{2n-k-1}{n-1}2^kk
=\binom{2n-\ell}{n}2^\ell n
\end{equation}
is valid, where an empty sum is understood to be $0$.
\end{lem}

\begin{proof}
We now prove the identity~\eqref{sum-central-binom-ell-eq} by the descending induction mentioned in~\cite[p.~88]{Serge-Lang-1996Group}.
\par
It is trivial that, when $n-\ell=0$, the identity~\eqref{sum-central-binom-ell-eq} is valid.
\par
Assume that the identity~\eqref{sum-central-binom-ell-eq} is valid for some $n-\ell\ge0$ with $\ell\ge1$. Then it is easy to see that
\begin{align*}
\sum_{k=\ell-1}^{n}\binom{2n-k-1}{n-1}2^kk
&=\sum_{k=\ell}^{n}\binom{2n-k-1}{n-1}2^kk+\binom{2n-\ell}{n-1}2^{\ell-1}(\ell-1)\\
&=\binom{2n-\ell}{n}2^\ell n+\binom{2n-\ell}{n-1}2^{\ell-1}(\ell-1)\\
&=\binom{2n-\ell+1}{n}2^{\ell-1}n.
\end{align*}
Inductively, we conclude that the identity~\eqref{sum-central-binom-ell-eq} is valid for all $n\ge\ell\ge0$.
The proof of Lemma~\ref{sum-central-binom-lem-ell} is thus complete.
\end{proof}

\begin{rem}
The proof of Lemma~\ref{sum-central-binom-lem-ell} is inspired by Darij Grinberg's suggestion of using descending induction at the website \url{https://mathoverflow.net/q/402822/} (accessed on 30 August 2021). The identity~\eqref{sum-central-binom-ell-eq} in Lemma~\ref{sum-central-binom-lem-ell} has been posted by the author at the sites \url{https://mathoverflow.net/a/402833/} (accessed on 30 August 2021) and \url{https://mathoverflow.net/a/403278/} (accessed on 6 September 2021).
\end{rem}

\begin{lem}\label{Stack-Quest-Answ-lem}
Let
\begin{equation}\label{P(n-k)-dfn-Notation}
P(n,k)=k!(2n-2k-1)!! \binom{2n-k-1}{k-1}, \quad n\ge k\in\mathbb{N}_0,
\end{equation}
where we assume that $\binom{-1}{-1}=1$ and $\binom{z}{-1}=0$ for $z\in\mathbb{Z}\setminus\{-1\}$. 
Then
\begin{equation}\label{Stack-Quest-Answ-Eq}
\sum_{m=0}^{k}(-1)^{m} \binom{k}{m}\binom{m/2}{\ell}
=
\begin{dcases}
0, & k>\ell\\
\frac{(-1)^{\ell}}{(2\ell)!!}P(\ell,k), & \ell\ge k
\end{dcases}
\end{equation}
for $k,\ell\in\mathbb{N}_0$.
\end{lem}

\begin{proof}
In~\cite[p.~165]{Quaintance-Gould-2016-B}, the equation~(12.1) reads that
\begin{equation}\label{(12.1)-Quaintance}
n!\binom{z}{n}=\sum_{\ell=0}^{n}s(n,\ell)z^\ell, \quad n\in\mathbb{N}_0, \quad z\in\mathbb{C},
\end{equation}
where $s(n,\ell)$ is defined by~\eqref{Stirl-No-First-GF}.
Making use of the relation~\eqref{(12.1)-Quaintance} and utilizing the explicit formula~\eqref{Stirling-Number-dfn}, we find
\begin{equation}
\begin{aligned}\label{Stack-binom-Qi-Ask}
\sum_{m=0}^{k}(-1)^{m} \binom{k}{m}\binom{m/2}{\ell}
&=\frac1{\ell!}\sum_{m=0}^{k}(-1)^{m} \binom{k}{m}\biggl[\ell!\binom{m/2}{\ell}\biggr]\\
&=\frac1{\ell!}\sum_{m=0}^{k}(-1)^{m} \binom{k}{m}\sum_{q=0}^{\ell}s(\ell,q)\biggl(\frac{m}{2}\biggr)^q\\
&=\frac1{\ell!}\sum_{q=0}^{\ell}\frac{s(\ell,q)}{2^q} \sum_{m=0}^{k}(-1)^{m} \binom{k}{m}m^q\\
&=(-1)^k\frac{k!}{\ell!}\sum_{q=0}^{\ell}s(\ell,q)\biggl(\frac{1}{2}\biggr)^q S(q,k)
\end{aligned}
\end{equation}
for $k,\ell\in\mathbb{N}_0$, where $0^0$ was regarded as $1$. Hence, it is easy to see that
\begin{equation*}
\sum_{m=0}^{k}(-1)^{m} \binom{k}{m}\binom{m/2}{\ell}=0, \quad 0\le\ell<k.
\end{equation*}
\par
In~\cite[Theorem~2.1]{CDM-68111.tex}, the identity
\begin{equation}\label{Bell-fall-Eq}
\bell_{n,k}(\langle\alpha\rangle_1, \langle\alpha\rangle_2, \dotsc,\langle\alpha\rangle_{n-k+1})
=\frac{(-1)^k}{k!}\sum_{\ell=0}^{k}(-1)^{\ell}\binom{k}{\ell}\langle\alpha\ell\rangle_n
\end{equation}
for $n\ge k\in\mathbb{N}_0$ and $\alpha\in\mathbb{C}$ was proved. Taking $\alpha=\frac{1}{2}$ in~\eqref{Bell-fall-Eq}, applying the relation
\begin{equation*}
\langle z\rangle_n=\sum_{\ell=0}^{n}s(n,\ell)z^\ell, \quad z\in\mathbb{C}, \quad n\in\mathbb{N}_0
\end{equation*}
in~\cite[p.~9, (1.26)]{Temme-96-book}, interchanging the order of sums, and employing~\eqref{Stirling-Number-dfn} yield
\begin{align*}
\bell_{n,k}\biggl(\biggl\langle\frac12\biggr\rangle_1, \biggl\langle\frac12\biggr\rangle_2,\dotsc, \biggl\langle\frac12\biggr\rangle_{n-k+1}\biggr)
&=\frac{(-1)^k}{k!}\sum_{\ell=0}^{k}(-1)^{\ell}\binom{k}{\ell}\biggl\langle\frac{\ell}{2}\biggr\rangle_n\\
&=\frac{(-1)^k}{k!}\sum_{\ell=0}^{k}(-1)^{\ell}\binom{k}{\ell}\sum_{m=0}^{n}s(n,m)\biggl(\frac{\ell}{2}\biggr)^m\\
&=\frac{(-1)^k}{k!}\sum_{m=0}^{n}s(n,m)\biggl(\frac{1}{2}\biggr)^m\sum_{\ell=0}^{k}(-1)^{\ell}\binom{k}{\ell}\ell^m\\
&=\sum_{m=0}^{n}s(n,m)\biggl(\frac{1}{2}\biggr)^mS(m,k)
\end{align*}
for $n\ge k\in\mathbb{N}_0$.
Combining this with~\eqref{Bell-Polyn-Half} results in
\begin{equation}\label{2Stirl-sum-prod-ID}
\sum_{m=k}^{n}s(n,m)\biggl(\frac{1}{2}\biggr)^mS(m,k)
=(-1)^{n+k}\frac{(2n-2k-1)!!}{2^n} \binom{2n-k-1}{k-1}
\end{equation}
for $n\ge k\in\mathbb{N}_0$.
Substituting the identity~\eqref{2Stirl-sum-prod-ID} into~\eqref{Stack-binom-Qi-Ask}, we find
\begin{align*}
\sum_{m=0}^{k}(-1)^{m} \binom{k}{m}\binom{m/2}{\ell}
&=(-1)^{\ell}k!\frac{(2\ell-2k-1)!!}{(2\ell)!!}\binom{2\ell-k-1}{k-1}\\
&=\frac{(-1)^{\ell}}{(2\ell)!!}P(\ell,k)
\end{align*}
for $\ell\ge k\in\mathbb{N}_0$. The proof of Lemma~\ref{Stack-Quest-Answ-lem} is thus complete.
\end{proof}

\begin{rem}
The identity~\eqref{Stack-Quest-Answ-Eq} was first announced by the author in the MathStackExchange answers at \url{https://math.stackexchange.com/a/4268339/} (accessed on 5 October 2021) and \url{https://math.stackexchange.com/a/4268341/} (accessed on 5 October 2021), and subsequently recovered and generalized in~\cite[Section~4]{2nd-Bell-Polyn-factoria-vl.tex}.
\end{rem}

Within the theory of Gauss hypergeometric functions ${\,}_2F_1$, the following Euler integral representation introduced in 1748 stands as a foundational and celebrated result. 

\begin{lem}[{\cite[p.~110, Eq.~(5.4)]{Temme-96-book}}]
If $\Re(c)>\Re(b)>0$ and $|\arg(1-z)|<\pi$, then
\begin{equation}\label{Euler-Integral-Gauss-HF}
{\,}_2F_1(a,b;c;z)=\frac{\Gamma(c)}{\Gamma(b)\Gamma(c-b)} \int_{0}^{1}\frac{t^{b-1}(1-t)^{c-b-1}}{(1-z t)^a}\td t
\end{equation}
is valid in the $z$ plane cut along the real axis from $1$ to $\infty$, where it is understood that $\arg t=\arg(1-t)=0$ and $\frac{1}{(1-z t)^a}$ has its principal value.
\end{lem}

Recall from~\cite[Chapter~XIII]{mpf-1993} and~\cite[Chapter~IV]{Widder-1941B} that an infinitely differentiable function $f$ is said to be completely (or absolutely, respectively) monotonic on an interval $I$ if it has derivatives of all orders on $I$ and satisfies $(-1)^{n}f^{(n)}(x)\ge 0$ (or $f^{(n)}(x)\ge 0$, respectively) for all $x\in I$ and $n\in\mathbb{N}_0$.
Recall from~\cite[Definition~1]{absolute-mon-simp.tex},~\cite[Definition~1]{compmon2}, and~\cite[Definition~5.10 and Comments~5.29]{Schilling-Song-Vondracek-2nd} that an infinitely differentiable and positive function $f$ is said to be logarithmically absolutely (or completely, respectively) monotonic on an interval $I$ if $[\ln f(x)]^{(n)}\ge 0$ (or $(-1)^{n}[\ln f(x)]^{(n)}\ge0$, respectively) for all $n\in\mathbb{N}$ and $x\in I$.
In~\cite[Theorem~1.1]{CBerg-2004},~\cite[Theorem~1]{absolute-mon-simp.tex},~\cite[Theorem~1]{compmon2}, and~\cite[p.~627, (1.4)]{JAAC384.tex}, the following two relations between these two pairs of functions were derived or reviewed:
\begin{enumerate}
\item
A logarithmically completely monotonic function on an interval $I$ is also completely monotonic on $I$, but not conversely.
\item
A logarithmically absolutely monotonic function on an interval $I$ is also absolutely monotonic on $I$, but not conversely.
\end{enumerate}
\par
If $f(x)$ is not identically zero and is completely monotonic on $(0,\infty)$, then $f^{(n)}(x)$ for $n\in\mathbb{N}_0$ cannot vanish at any point of $(0,\infty)$; see~\cite{Dubourdieu96} and~\cite[p.~395]{haerc1}.
\par
It is easy to see that a function $f(x)$ is (logarithmically) completely monotonic on an interval $(a,b)$ if and only if $f(-x)$ is (logarithmically) absolutely monotonic on the interval $(-b,-a)$; see~\cite[p.~145, Definition~2c]{Widder-1941B}. These concepts and conclusions will be employed in Section~\ref{coefficients-analysis-sec}.

\section{Taylor power series expansions of inverse tangent function}\label{arctan-taylor-maclaurin-sec}

In this section, by virtue of some preparations in Section~\ref{preparation-sec}, we now present the Maclaurin and the Taylor power series expansions of the four functions in~\eqref{four-arctangent}.

\begin{thm}\label{arctan-pi-4-ser-thm}
For $n\in\mathbb{N}$, the $n$th derivative of $\arctan z$ is
\begin{equation*}
(\arctan z)^{(n)}
=\frac{(n-1)!}{(2z)^{n-1}}\sum_{k=0}^{n-1}(-1)^k\binom{k}{n-k-1}\frac{(2z)^{2k}}{(1+z^2)^{k+1}}.
\end{equation*}
The function $\frac{\arctan z}{z}$ can be expanded into the Taylor power series
\begin{equation}\label{arctan-pi-4-ser-eq}
\frac{\arctan z}{z}=\sum_{n=0}^{\infty}(-1)^n\biggl[\frac{\pi}{4}+T(n)\biggr](z-1)^n,\quad |z-1|<\sqrt{2}\,,
\end{equation}
where
\begin{equation}\label{T(n)-dfn-notation}
T(n)=
\begin{dcases}
0, & n=0;\\
\sum_{k=1}^{n}\frac{(-1)^{k}}{k}\frac{\sin\frac{3k\pi}{4}}{2^{k/2}}, & n\in\mathbb{N}.
\end{dcases}
\end{equation}
\end{thm}

\begin{proof}
For $n\in\mathbb{N}$, by virtue of the Fa\`a di Bruno formula~\eqref{Bruno-Bell-Polynomial}, we arrive at
\begin{align*}
(\arctan z)^{(n)}&=\biggl(\frac1{1+z^2}\biggr)^{(n-1)}\\
&=\sum_{k=0}^{n-1}\frac{(-1)^kk!}{(1+z^2)^{k+1}}\bell_{n-1,k}(2z,2,0,\dotsc,0)\\
&=\sum_{k=0}^{n-1}\frac{(-1)^kk!}{(1+z^2)^{k+1}}2^k\bell_{n-1,k}(z,1,0,\dotsc,0)\\
&=\sum_{k=0}^{n-1}\frac{(-1)^kk!}{(1+z^2)^{k+1}}2^k
\frac{(n-k-1)!}{2^{n-k-1}}\binom{n-1}{k}\binom{k}{n-k-1}z^{2k-n+1}\\
&=\frac{(n-1)!}{(2z)^{n-1}}\sum_{k=0}^{n-1}(-1)^k\binom{k}{n-k-1}\frac{(2z)^{2k}}{(1+z^2)^{k+1}},
\end{align*}
where we also utilized the identities~\eqref{Bell(n-k)} and~\eqref{Bell-x-1-0-eq}. Accordingly, for $n\in\mathbb{N}_0$, we have
\begin{align*}
&\biggl(\frac{\arctan z}{z}\biggr)^{(n)}=\sum_{k=0}^{n}\binom{n}{k}(\arctan z)^{(k)} \biggl(\frac1z\biggr)^{(n-k)}\\
&=(-1)^n\frac{n!}{z^{n}} \Biggl[\frac{\arctan z}{z}
+\sum_{k=1}^{n}\frac{(-1)^{k}}{k}\frac{1}{2^{k-1}} \sum_{\ell=0}^{k-1}(-1)^\ell\binom{\ell}{k-\ell-1}\frac{2^{2\ell}z^{2\ell}}{(1+z^2)^{\ell+1}}\Biggr]\\
&\to(-1)^nn!\Biggl[\frac{\pi}{4} +\sum_{k=1}^{n}\frac{(-1)^{k}}{2^kk} \sum_{\ell=0}^{k-1}(-2)^\ell\binom{\ell}{k-\ell-1}\Biggr]\\
&=(-1)^nn!\biggl[\frac{\pi}{4}+T(n)\biggr]
\end{align*}
as $z\to1$, where we used the identity
\begin{multline}\label{Cos-sincomb-ID}
\sum_{\ell=0}^{k}(-2)^\ell\binom{\ell}{k-\ell}
=(-2)^k\sum_{\ell=0}^{\lfloor k/2\rfloor}\biggl(-\frac{1}{2}\biggr)^{\ell}\binom{k-\ell}{\ell}\\
=2^{k/2}\biggl(\cos\frac{3k\pi}{4}-\sin\frac{3k\pi}{4}\biggr)
=2^{(k+1)/2}\sin\frac{3(k+1)\pi}{4}
\end{multline}
for $k\in\mathbb{N}_0$, which can be deduced from taking $z=-\frac12$ in the identity
\begin{equation}\label{Riordan(1.70)B}
\sum_{k=0}^{\lfloor n/2\rfloor}\binom{n-k}{k}z^k
=\frac{1}{\sqrt{1+4z}\,} \biggl[\biggl(\frac{1+\sqrt{1+4z}\,}{2}\biggr)^{n+1} -\biggl(\frac{1-\sqrt{1+4z}\,}{2}\biggr)^{n+1}\biggr]
\end{equation}
or
\begin{equation}\label{Riordan(1.71)B}
\sum_{k=0}^{\lfloor n/2\rfloor}\binom{n-k}{k}z^k
=\frac{x^{n+1}-1}{(x-1)(1+x)^n}, \quad z=-\frac{x}{(1+x)^2}
\end{equation}
collected as~(1.70) and~(1.71) on~\cite[p.~33]{Sprugnoli-Gould-2006}, where $\lfloor{x}\rfloor$ denotes the floor function whose value is equal to the largest integer less than or equal to $x$.
Consequently, we obtain the Taylor power series expansion~\eqref{arctan-pi-4-ser-eq}.
\par
The unremovable singular points of $\frac{\arctan z}{z}$ closest to the point $z=1$ are $\pm\ti$. Therefore, the disc of convergence of the Taylor power series expansion~\eqref{arctan-pi-4-ser-eq} is $|z-1|<\sqrt{2}\,$.
The proof of Theorem~\ref{arctan-pi-4-ser-thm} is complete.
\end{proof}

\begin{rem}
The first few values of the sequence $T(n)$ for $0\le n\le9$ are given by
\begin{equation*}
0,\quad  -\frac{1}{2},\quad  -\frac{3}{4},\quad  -\frac{5}{6},\quad  -\frac{5}{6},\quad  -\frac{97}{120},\quad  -\frac{63}{80},\quad  -\frac{109}{140},\quad  -\frac{109}{140},\quad  -\frac{7883}{10080}.
\end{equation*}
These values suggest that $T(n)<0$ and $T(4n-1)=T(4n)$ for $n\in\mathbb{N}$.
\par
Taking the limit $z\to0$ on both sides of the series expansion~\eqref{arctan-pi-4-ser-eq} in Theorem~\ref{arctan-pi-4-ser-thm} shows
\begin{equation*}
\sum_{n=0}^{\infty}\biggl[\frac{\pi}{4}+T(n)\biggr]=1.
\end{equation*}
By the necessary condition for a series to be convergent, we obtain
\begin{equation}\label{T(n)lim-to-Pi}
\lim_{n\to\infty}T(n)=\sum_{k=1}^{\infty}\frac{(-1)^{k}}{k}\frac{\sin\frac{3k\pi}{4}}{2^{k/2}}
=-\frac{\pi}{4}.
\end{equation}
We can regard~\eqref{T(n)lim-to-Pi} as a series representation of the circular constant $\pi$.
\end{rem}

\begin{rem}
In addition to~\eqref{Riordan(1.70)B} and~\eqref{Riordan(1.71)B}, the identities
\begin{align*}
\sum_{k=0}^{\lfloor n/2\rfloor}(-1)^k\binom{n-k}{k}(xy)^k(x+y)^{n-2k}
&=\frac{x^{n+1}-y^{n+1}}{x-y},\\
\sum_{k=0}^{\lfloor n/2\rfloor}(-1)^k\binom{n-k}{k}(2\cos x)^{n-2k}
&=\frac{\sin[(n+1)x]}{\sin x},\quad n\ge-1,
\end{align*}
and
\begin{equation*}
\sum_{k=0}^{\lfloor n/2\rfloor}\binom{n-k}{k}z^k2^{n-2k}
=\frac{x^{n+1}-y^{n+1}}{x-y}, \quad
\begin{dcases}
x=1+\sqrt{z+1}\,;\\
y=1-\sqrt{z+1}\,
\end{dcases}
\end{equation*}
can also be used to deduce the identity~\eqref{Cos-sincomb-ID}. These identities are modified versions of~\cite[pp.~30--31, Eqs.~(1.60)--(1.62)]{Sprugnoli-Gould-2006}.
\end{rem}

\begin{cor}\label{arctan-frac-ser-mid-cor}
For $|z-1|<1$, the function $\frac{\arctan z-\frac{\pi}{4}}{z}$ has the Taylor power series expansion
\begin{equation}\label{arctan-frac-ser-mid}
\frac{\arctan z-\frac{\pi}{4}}{z}=\sum_{n=1}^{\infty}(-1)^nT(n)(z-1)^n,
\end{equation}
where $T(n)$ is defined by~\eqref{T(n)-dfn-notation}.
\end{cor}

\begin{proof}
The Taylor power series expansion~\eqref{arctan-pi-4-ser-eq} can be reformulated as
\begin{equation*}
\frac{\arctan z}{z}=\frac{\pi}{4}\sum_{n=0}^{\infty}(1-z)^n+\sum_{n=0}^{\infty}T(n)(1-z)^n
=\frac{\pi}{4z}+\sum_{n=1}^{\infty}T(n)(1-z)^n
\end{equation*}
for $|z-1|<1$. The proof of Corollary~\ref{arctan-frac-ser-mid-cor} is complete.
\end{proof}

\begin{cor}\label{cor3.2}
For $|z|<\sqrt{2}\,$, the function $\arctan(1-z)$ has the Maclaurin power series expansion
\begin{align*}
\arctan (1-z)&=\frac{\pi}{4}+\sum_{n=1}^{\infty}\frac{(-1)^{n}}{2^{n/2}n}\sin\frac{3n\pi}{4}z^n\\
&=\frac{\pi}{4}-\frac{z}{2}-\frac{z^2}{4}-\frac{z^3}{12}+\frac{z^5}{40}+\frac{z^6}{48}+\frac{z^7}{112}-\frac{z^9}{288}-\dotsm.
\end{align*}
\end{cor}

\begin{proof}
This follows from replacing $z$ by $1-z$ in the Taylor power series expansion~\eqref{arctan-pi-4-ser-eq} and simplifying.
\end{proof}

\begin{rem}
The proof of Corollary~\ref{cor3.2} can also be found in the answer by the author at the website \url{https://math.stackexchange.com/a/4332741/} (accessed on 14 December 2021).
\end{rem}

\begin{rem}
Letting $z=\sqrt{3}\,$ on both sides of the Taylor power series expansion~\eqref{arctan-frac-ser-mid} results in a series representation
\begin{equation*}
\frac{\pi}{12\sqrt{3}\,}=\sum_{n=1}^{\infty}(-1)^{n}T(n)\bigl(\sqrt{3}\,-1\bigr)^n,
\end{equation*}
where $T(n)$ is defined by~\eqref{T(n)-dfn-notation}.
\end{rem}

\begin{thm}\label{arctan-sqrt-ser-expan-thm}
The function $\frac{\arctan\sqrt{z}\,}{\sqrt{z}\,}$ has the Taylor power series expansion
\begin{multline}\label{arctan-sqrt-expan-eq}
\frac{\arctan\sqrt{z}\,}{\sqrt{z}\,}\\
=\sum_{n=0}^{\infty} \frac{(-1)^n}{2^n}\Biggl[\frac{(2n-1)!!\pi}{4}
+\frac{n!}{2^{n}}\sum_{k=1}^{n} (-1)^{k}\binom{2n-k}{n} \frac{2^{k/2}}{k}\sin\frac{3k\pi}{4}\Biggr] \frac{(z-1)^n}{n!}
\end{multline}
for $|z-1|<2$, where $(-1)!!=1$ is defined by~\eqref{double-factorial-neg-odd-integer}.
\end{thm}

\begin{proof}
Let $w=f(z)=\sqrt{z}\,\to1$ as $z\to1$. In light of the Fa\`a di Bruno formula~\eqref{Bruno-Bell-Polynomial}, in view of the identity~\eqref{Bell(n-k)}, and by virtue of~\eqref{Bell-Polyn-Half}, we derive
\begin{align*}
\biggl(\frac{\arctan\sqrt{z}\,}{\sqrt{z}\,}\biggr)^{(n)}
&=\sum_{k=0}^{n} \biggl(\frac{\arctan w}{w}\biggr)^{(k)}\\
&\quad\times\bell_{n,k}\biggl(\biggl\langle\frac{1}{2}\biggr\rangle_1\frac{1}{z^{1/2}},\biggl\langle\frac{1}{2}\biggr\rangle_2\frac{1}{z^{3/2}},\dotsc, \biggl\langle\frac{1}{2}\biggr\rangle_{n-k+1}\frac{1}{z^{[2(n-k)+1]/2}}\biggr)\\
&=\sum_{k=0}^{n} \biggl(\frac{\arctan w}{w}\biggr)^{(k)} z^{k/2-n} \bell_{n,k}\biggl(\biggl\langle\frac{1}{2}\biggr\rangle_1,\biggl\langle\frac{1}{2}\biggr\rangle_2, \dotsc, \biggl\langle\frac{1}{2}\biggr\rangle_{n-k+1}\biggr)\\
&=\sum_{k=0}^{n} \biggl(\frac{\arctan w}{w}\biggr)^{(k)} z^{k/2-n} (-1)^{n+k}\frac{(2n-2k-1)!!}{2^n} \binom{2n-k-1}{k-1}\\
&\to
\frac{(-1)^{n}}{2^n}\sum_{k=0}^{n} P(n,k)\biggl[\frac{\pi}{4}+T(k)\biggr]
\end{align*}
as $z\to1$ for $n\in\mathbb{N}_0$, where in the last step we used the series expansion~\eqref{arctan-pi-4-ser-eq} in Theorem~\ref{arctan-pi-4-ser-thm}
and used the notation $P(n,k)$ and $T(k)$ defined by~\eqref{P(n-k)-dfn-Notation} and~\eqref{T(n)-dfn-notation}, respectively. Consequently, we obtain the Taylor power series expansion
\begin{equation}\label{arctan-sqrt-ser-expan-eq}
\frac{\arctan\sqrt{z}\,}{\sqrt{z}\,}=\sum_{n=0}^{\infty} (-1)^n\Biggl(\sum_{k=0}^{n} P(n,k) \biggl[\frac{\pi}{4}+T(k)\biggr]\Biggr) \frac{(z-1)^n}{(2n)!!}, \quad |z-1|<2.
\end{equation}
The unremovable singular point of the function $\frac{\arctan\sqrt{z}\,}{\sqrt{z}\,}$ closest to the point $z=1$ is $z=-1$. Accordingly, the disc of convergence of the series in~\eqref{arctan-sqrt-ser-expan-eq} on the complex plane $\mathbb{C}$ is $|z-1|<2$.
\par
Replacing $\sqrt{z}\,$ by $z$ on both sides of~\eqref{arctan-sqrt-ser-expan-eq} leads to
\begin{equation}\label{arctan-sqrt-ser-cor-eq}
\frac{\arctan z}{z}=\sum_{n=0}^{\infty} (-1)^n\Biggl(\sum_{k=0}^{n} P(n,k) \biggl[\frac{\pi}{4}+T(k)\biggr]\Biggr) \frac{\bigl(z^2-1\bigr)^n}{(2n)!!},\quad |z^2-1|<2.
\end{equation}
Comparing the series expansion~\eqref{arctan-pi-4-ser-eq} with the series expansion~\eqref{arctan-sqrt-ser-cor-eq} shows
\begin{equation*}
\sum_{n=0}^{\infty}(-1)^n\biggl[\frac{\pi}{4}+T(n)\biggr](z-1)^n
=\sum_{n=0}^{\infty} (-1)^n\Biggl(\sum_{k=0}^{n} P(n,k) \biggl[\frac{\pi}{4}+T(k)\biggr]\Biggr) \frac{\bigl(z^2-1\bigr)^n}{(2n)!!}.
\end{equation*}
Differentiating $q\in\mathbb{N}_0$ times with respect to $z$ on both sides of this equation gives
\begin{align*}
&\quad\sum_{n=q}^{\infty}(-1)^n\biggl[\frac{\pi}{4}+T(n)\biggr]\langle n\rangle_q(z-1)^{n-q}\\
&=\sum_{n=0}^{\infty} \frac{(-1)^n}{(2n)!!}\Biggl(\sum_{k=0}^{n} P(n,k) \biggl[\frac{\pi}{4}+T(k)\biggr]\Biggr) \bigl[\bigl(z^2-1\bigr)^n\bigr]^{(q)}\\
&=\sum_{n=0}^{\infty} \frac{(-1)^n}{(2n)!!}\Biggl(\sum_{k=0}^{n} P(n,k) \biggl[\frac{\pi}{4}+T(k)\biggr]\Biggr)\\
&\quad\times\sum_{p=0}^{q}(2z)^{2p-q}(q-p)!\binom{q}{p}\binom{p}{q-p}\langle n\rangle_p \bigl(z^2-1\bigr)^{n-p}\\
&=\sum_{p=0}^{q}(2z)^{2p-q}(q-p)!\binom{q}{p}\binom{p}{q-p}\\
&\quad\times\sum_{n=0}^{\infty} \frac{(-1)^n}{(2n)!!}\Biggl(\sum_{k=0}^{n} P(n,k) \biggl[\frac{\pi}{4}+T(k)\biggr]\Biggr) \langle n\rangle_p \bigl(z^2-1\bigr)^{n-p},
\end{align*}
where we used the derivative formula
\begin{align*}
\bigl[\bigl(z^2-1\bigr)^n\bigr]^{(q)}&=\sum_{p=0}^{q} \bigl(u^n\bigr)^{(p)}\bell_{q,p}(2z,2,0,\dotsc,0), \quad u=u(z)=z^2-1\\
&=\sum_{p=0}^{q}\langle n\rangle_p \bigl(z^2-1\bigr)^{n-p}2^p\bell_{q,p}(z,1,0,\dotsc,0)\\
&=\sum_{p=0}^{q}(2z)^{2p-q}(q-p)!\binom{q}{p}\binom{p}{q-p}\langle n\rangle_p \bigl(z^2-1\bigr)^{n-p}
\end{align*}
for $q\in\mathbb{N}_0$, in which we employed the formulas~\eqref{Bruno-Bell-Polynomial}, \eqref{Bell(n-k)}, and~\eqref{Bell-x-1-0-eq}.
Furthermore, letting $z=1$ and simplifying reveal
\begin{align*}
(-1)^qq!\biggl[\frac{\pi}{4}+T(q)\biggr]
&=\sum_{p=0}^{q}(-1)^p2^{p-q}(q-p)!\binom{q}{p}\binom{p}{q-p} \sum_{k=0}^{p} P(p,k) \biggl[\frac{\pi}{4}+T(k)\biggr]\\
&=\sum_{k=0}^{q} \Biggl[\sum_{p=k}^{q}(-1)^p2^{p-q}(q-p)!\binom{q}{p}\binom{p}{q-p}P(p,k)\Biggr] \biggl[\frac{\pi}{4}+T(k)\biggr]\\
&=\frac{\pi}{4}\sum_{k=0}^{q} \Biggl[\sum_{p=k}^{q}(-1)^p2^{p-q}(q-p)!\binom{q}{p}\binom{p}{q-p}P(p,k)\Biggr]\\
&\quad+\sum_{k=0}^{q} \Biggl[\sum_{p=k}^{q}(-1)^p2^{p-q}(q-p)!\binom{q}{p}\binom{p}{q-p}P(p,k)\Biggr]T(k)
\end{align*}
for $q\in\mathbb{N}_0$. Further equating the rational and irrational parts on both sides yields
\begin{equation*}
\sum_{k=0}^{q}\sum_{p=k}^{q}(-1)^p2^{p-q}(q-p)!\binom{q}{p}\binom{p}{q-p}P(p,k)
=(-1)^qq!
\end{equation*}
and
\begin{equation*}
\sum_{k=0}^{q} \Biggl[\sum_{p=k}^{q}(-1)^p2^{p-q}(q-p)!\binom{q}{p}\binom{p}{q-p}P(p,k)\Biggr]T(k)
=(-1)^qq!T(q)
\end{equation*}
for $q\in\mathbb{N}_0$. Replacing $q$ by $n$ and simplifying result in
\begin{equation}\label{Curious-id-ID-O}
\sum_{k=0}^{n}\sum_{\ell=k}^{n}\frac{(-2)^{\ell}}{\ell!}\binom{\ell}{n-\ell}P(\ell,k)
=\sum_{\ell=0}^{n}\binom{\ell}{n-\ell}\frac{(-2)^{\ell}}{\ell!}\sum_{k=0}^{\ell}P(\ell,k)
=(-2)^n
\end{equation}
and
\begin{equation}\label{Curious-id-ID-T}
\begin{aligned}
\sum_{k=0}^{n} \Biggl[\sum_{\ell=k}^{n}\binom{\ell}{n-\ell}\frac{(-2)^{\ell}}{\ell!}P(\ell,k)\Biggr]T(k)
&=\sum_{\ell=0}^{n}\binom{\ell}{n-\ell} \frac{(-2)^{\ell}}{\ell!}\sum_{k=0}^{\ell}P(\ell,k)T(k)\\
&=(-2)^nT(n)
\end{aligned}
\end{equation}
for $n\in\mathbb{N}_0$.
\par
For $k\in\mathbb{N}$, let $s_k$ and $S_k$ be two sequences independent of $n$ such that $n\ge k\in\mathbb{N}$. The inversion theorem in~\cite[p.~528, Theorem~4.4]{AADM-3031.tex} reads that
\begin{equation}\label{Qi-Zou-Guo-Inversion-S}
(-1)^nnS_n=\sum_{k=1}^{n}\binom{2n-k-1}{n-1}(-1)^kks_k
\end{equation}
if and only if
\begin{equation}\label{Qi-Zou-Guo-Inversion-s}
s_n=\sum_{k=1}^{n}\binom{k}{n-k}S_k.
\end{equation}
Applying the inversion theorem in~\eqref{Qi-Zou-Guo-Inversion-S} and~\eqref{Qi-Zou-Guo-Inversion-s} to the second identities in~\eqref{Curious-id-ID-O} and~\eqref{Curious-id-ID-T}, respectively, we obtain
\begin{equation*}
\sum_{\ell=0}^{n}P(n,\ell)=\frac{(n-1)!}{2^{n}}\sum_{k=1}^{n}\binom{2n-k-1}{n-1}2^kk
=\frac{(n-1)!}{2^{n}}\binom{2n}{n}n=(2n-1)!!
\end{equation*}
and
\begin{equation*}
(-1)^nn \frac{(-2)^{n}}{n!}\sum_{\ell=0}^{n}P(n,\ell)T(\ell)=\sum_{k=1}^{n} \binom{2n-k-1}{n-1}(-1)^k k(-2)^kT(k)
\end{equation*}
for $n\in\mathbb{N}$, where we used the identity~\eqref{sum-central-binom-ell-eq} for $\ell=0$ in Lemma~\ref{sum-central-binom-lem-ell}.
In a word, the identities
\begin{equation}\label{Curious-id-inv-O}
\sum_{k=0}^{n}P(n,k)=(2n-1)!!
\end{equation}
and
\begin{equation}\label{Curious-id-inv-T}
\sum_{k=0}^{n}P(n,k)T(k)=\frac{(n-1)!}{2^{n}}\sum_{k=1}^{n} \binom{2n-k-1}{n-1}k2^kT(k)
\end{equation}
are valid for $n\in\mathbb{N}$.
Applying~\eqref{T(n)-dfn-notation} to~\eqref{Curious-id-inv-T} gives
\begin{equation}\label{Curious-id-inv-TT}
\begin{aligned}
\sum_{\ell=0}^{n}P(n,\ell)T(\ell)
&=\frac{(n-1)!}{2^{n}}\sum_{\ell=1}^{n}\Biggl[\sum_{k=\ell}^{n} \binom{2n-k-1}{n-1}2^kk\Biggr] \frac{(-1)^{\ell}}{\ell} \frac{1}{2^{\ell/2}}\sin\frac{3\ell\pi}{4}\\
&=\frac{(n-1)!}{2^{n}}\sum_{\ell=1}^{n}\binom{2n-\ell}{n}2^\ell n \frac{(-1)^{\ell}}{\ell} \frac{1}{2^{\ell/2}} \sin\frac{3\ell\pi}{4}\\
&=\frac{n!}{2^{n}}\sum_{k=1}^{n}(-1)^{k}\binom{2n-k}{n}\frac{2^{k/2}}{k} \sin\frac{3k\pi}{4},
\end{aligned}
\end{equation}
where we used the identity~\eqref{sum-central-binom-ell-eq} for $\ell\ge1$ in Lemma~\ref{sum-central-binom-lem-ell}.
\par
Substituting the identities~\eqref{Curious-id-inv-O} and~\eqref{Curious-id-inv-TT} into~\eqref{arctan-sqrt-ser-expan-eq} yields the series expansion~\eqref{arctan-sqrt-expan-eq}.
The proof of Theorem~\ref{arctan-sqrt-ser-expan-thm} is complete.
\end{proof}

\begin{rem}
In the papers~\cite{Fretigny-Qi-Gauss.tex, Gauss-hyperg-Int.tex, Qi-Izan-Peraz-Gauss.tex}, the higher derivatives of the functions $\frac{\arctan z}{z}$ and $\frac{\arctan\sqrt{z}\,}{\sqrt{z}\,}$ were further discussed and investigated. Due to the space limitations of this paper, we do not provide a detailed description of them again.
\end{rem}

\begin{rem}
Setting $z=\frac{1}{3}$ on both sides of the series expansion~\eqref{arctan-sqrt-expan-eq} in Theorem~\ref{arctan-sqrt-ser-expan-thm} and simplifying lead to a series representation
\begin{equation*}
\frac{\pi}{4\sqrt{3}\,}=\sum_{n=0}^{\infty}\frac{1}{6^n}\sum_{k=1}^{n} (-1)^{k+1} \binom{2n-k}{n} \frac{2^{k/2}}{k}\sin\frac{3k\pi}{4}.
\end{equation*}
\end{rem}

\begin{rem}
Let $s_n=(-2)^n$ and
\begin{equation*}
S_n=(-1)^n2\binom{2n-1}{n}
\end{equation*}
for $n\in\mathbb{N}$. Applying these two sequences to the inversion theorem recited in~\eqref{Qi-Zou-Guo-Inversion-S} and~\eqref{Qi-Zou-Guo-Inversion-s} and considering the identity~\eqref{sum-central-binom-ell-eq} for $\ell=1$ lead to a combinatorial identity
\begin{equation}\label{sine-sum-id-one}
\sum_{k=1}^{n}(-1)^k\binom{k}{n-k}\binom{2k-1}{k}=(-1)^n2^{n-1}, \quad n\in\mathbb{N}.
\end{equation}
This identity has been posted by the author at the website \url{https://mathoverflow.net/a/405812/} (accessed on 9 September 2021).
\par
Let $S_n=(-2)^n$ and
\begin{equation*}
s_n=2^{(n+1)/2}\sin\frac{3(n+1)\pi}{4}
\end{equation*}
for $n\in\mathbb{N}$. Applying these two sequences to the inversion theorem restated in~\eqref{Qi-Zou-Guo-Inversion-S} and~\eqref{Qi-Zou-Guo-Inversion-s} and utilizing the identity~\eqref{Cos-sincomb-ID} reveal
\begin{equation}\label{sine-sum-id-two}
\sum_{k=1}^{n}(-1)^k\binom{2n-k-1}{n-1}2^{(k+1)/2}k\sin\frac{3(k+1)\pi}{4}=2^nn, \quad n\in\mathbb{N}.
\end{equation}
The identities~\eqref{sine-sum-id-one} and~\eqref{sine-sum-id-two} have been recovered and generalized in~\cite[Section~4]{2nd-Bell-Polyn-factoria-vl.tex}.
\end{rem}

\begin{rem}
Among other things, the identity~\eqref{Curious-id-inv-O} was also derived in~\cite[Theorem~4.2]{2nd-Bell-Polyn-factoria-vl.tex} and in~\cite[Remark~5.3]{Li-Qi-Authorea-2020}.
\end{rem}

\begin{thm}\label{sqrt-exp-wilf-ser-thm}
The function $\sqrt{2\te^{-z}-1}\,$ has the Maclaurin power series expansion
\begin{equation}\label{sqrt-exp-wilf-ser-eq}
\sqrt{2\te^{-z}-1}\, =\sum_{n=0}^{\infty}(-1)^{n+1} \Biggl[\sum_{k=0}^{n} (-1)^{k}S(n,k)(2k-3)!!\Biggr] \frac{z^n}{n!}, \quad |z|<\ln2,
\end{equation}
where $S(n,k)$ denotes the Stirling numbers of the second kind in~\eqref{2stirl-gen-f} and the double factorial of negative odd integers $-(2k+1)$ for $k\in\mathbb{N}_0$ is defined by~\eqref{double-factorial-neg-odd-integer}.
\end{thm}

\begin{proof}
Let $w=g(z)=2\te^{-z}-1$. Then
\begin{equation*}
\sqrt{w}\,=\sqrt{g(z)}\,=\sqrt{2\te^{-z}-1}\,\to1, \quad z\to0.
\end{equation*}
By the Fa\`a di Bruno formula~\eqref{Bruno-Bell-Polynomial}, we obtain
\begin{align*}
\bigl(\sqrt{2\te^{-z}-1}\,\bigr)^{(n)}&=\sum_{k=0}^{n}\bigl(\sqrt{w}\,\bigr)^{(k)} \bell_{n,k}\bigl(g'(z),g''(z),\dotsc,g^{(n-k+1)}(z)\bigr)\\
&=\sum_{k=0}^{n}\biggl\langle\frac{1}{2}\biggr\rangle_{k}w^{1/2-k} \bell_{n,k}\bigl(-2\te^{-z}, 2\te^{-z}, \dotsc, (-1)^{n-k+1}2\te^{-z}\bigr)\\
&=\sum_{k=0}^{n}\biggl\langle\frac{1}{2}\biggr\rangle_{k}\bigl(2\te^{-z}-1\bigr)^{1/2-k} (-1)^n\bigl(2\te^{-z}\bigr)^k
\bell_{n,k}(\underbrace{1, 1, \dotsc, 1}_{n-k+1})\\
&=(-1)^n\sum_{k=0}^{n}S(n,k) (-1)^{k-1}(2k-3)!! \bigl(2\te^{-z}-1\bigr)^{1/2-k} \te^{-k z}\\
&\to(-1)^n\sum_{k=0}^{n} (-1)^{k-1}S(n,k)(2k-3)!!
\end{align*}
as $z\to0$, where we employed the identities~\eqref{Bell(n-k)} and~\eqref{Bell-stirling}. The series expansion~\eqref{sqrt-exp-wilf-ser-eq} is thus proved. The proof of Theorem~\ref{sqrt-exp-wilf-ser-thm} is complete.
\end{proof}

\begin{thm}\label{sqrt-g(x)=0-bell-thm}
For $n\ge k\in\mathbb{N}_0$, the partial Bell polynomials $\bell_{n,k}$ satisfy
\begin{multline}\label{bell-polyn-sqrt-exp}
\bell_{n,k}\Biggl(1, 0, 1, 3, 16, 105, \dotsc, (-1)^{n-k}\sum_{\ell=0}^{n-k+1} (-1)^{\ell-1}S(n-k+1,\ell)(2\ell-3)!!\Biggr)\\
=\frac{(-1)^{n}}{k!}\sum_{\ell=k}^{n}(-1)^{\ell}S(n,\ell)P(\ell,k),
\end{multline}
where $P(\ell,k)$ is given by~\eqref{P(n-k)-dfn-Notation} and the Stirling numbers of the second kind $S(n,\ell)$ are generated by~\eqref{2stirl-gen-f}.
\end{thm}

\begin{proof}
Making use of the formula~\eqref{113-final-formula} yields
\begin{equation}\label{Bell-Polyn-final-formula}
\bell_{n+k,k}(x_1,x_2,\dotsc,x_{n+1})
=\binom{n+k}{k}\lim_{z\to0}\frac{\operatorname{d}^{n}}{\td z^{n}}\Biggl[\sum_{m=0}^{\infty} \frac{x_{m+1}}{(m+1)!}z^{m}\Biggr]^k.
\end{equation}
Setting $x_m=\bigl[\sqrt{g(z)}\,\bigr]^{(m)}\big|_{z=0}=\bigl(\sqrt{2\te^{-z}-1}\,\bigr)^{(m)}\big|_{z=0}$ for $m\in\mathbb{N}$ in~\eqref{Bell-Polyn-final-formula} leads to
\begin{align*}
&\quad\bell_{n+k,k}\Bigl(\bigl[\sqrt{g(z)}\,\bigr]'\big|_{z=0}, \bigl[\sqrt{g(z)}\,\bigr]''\big|_{z=0}, \dotsc,\bigl[\sqrt{g(z)}\,\bigr]^{(n+1)}\big|_{z=0}\Bigr)\\
&=\binom{n+k}{k}\lim_{z\to0}\frac{\operatorname{d}^{n}}{\td z^{n}}\Biggl[\sum_{m=0}^{\infty} \frac{\bigl[\sqrt{g(z)}\,\bigr]^{(m+1)}\big|_{z=0}}{(m+1)!}z^{m}\Biggr]^k\\
&=\binom{n+k}{k}\lim_{z\to0}\frac{\operatorname{d}^{n}}{\td z^{n}}\Biggl[\frac{1}{z}\sum_{m=1}^{\infty} \frac{\bigl[\sqrt{g(z)}\,\bigr]^{(m)}\big|_{z=0}}{m!}z^{m}\Biggr]^k\\
&=\binom{n+k}{k}\lim_{z\to0}\frac{\operatorname{d}^{n}}{\td z^{n}}\Biggl[\frac{\sqrt{g(z)}\,-1}{z}\Biggr]^k\\
&=\binom{n+k}{k}\lim_{z\to0}\frac{\operatorname{d}^{n}}{\td z^{n}}\Biggl[\frac{1}{z^k}\sum_{\ell=0}^{k}(-1)^{k-\ell}\binom{k}{\ell}g^{\ell/2}(z)\Biggr]\\
&=\binom{n+k}{k}\lim_{z\to0}\frac{\operatorname{d}^{n}}{\td z^{n}} \Biggl[\frac{1}{z^k}\sum_{\ell=0}^{k}(-1)^{k-\ell} \binom{k}{\ell}\bigl(2\te^{-z}-1\bigr)^{\ell/2}\Biggr]\\
&=\binom{n+k}{k}\lim_{z\to0}\frac{\operatorname{d}^{n}}{\td z^{n}} \Biggl[\frac{1}{z^k}\sum_{\ell=0}^{k}(-1)^{k-\ell} \binom{k}{\ell}\sum_{q=0}^{\infty}\binom{\ell/2}{q}\bigl(2\te^{-z}-2\bigr)^{q}\Biggr]\\
&=\binom{n+k}{k}\lim_{z\to0}\frac{\operatorname{d}^{n}}{\td z^{n}} \Biggl(\frac{1}{z^k} \sum_{\ell=0}^{k}(-1)^{k-\ell} \binom{k}{\ell}\sum_{q=0}^{\infty}(2q)!!\binom{\ell/2}{q} \sum_{p=q}^{\infty}(-1)^p S(p,q)\frac{z^p}{p!}\Biggr)\\
&=\binom{n+k}{k}\lim_{z\to0}\frac{\operatorname{d}^{n}}{\td z^{n}} \sum_{\ell=0}^{k}(-1)^{k-\ell} \binom{k}{\ell}\sum_{p=0}^{\infty}(-1)^p \Biggl[\sum_{q=0}^{p}(2q)!!\binom{\ell/2}{q} S(p,q)\Biggr]\frac{z^{p-k}}{p!}\\
&=\binom{n+k}{k} \sum_{\ell=0}^{k}(-1)^{k-\ell} \binom{k}{\ell}\sum_{p=0}^{\infty}(-1)^p \Biggl[\sum_{q=0}^{p}(2q)!!\binom{\ell/2}{q} S(p,q)\Biggr]\langle p-k\rangle_n\\
&\quad\times\lim_{z\to0}\frac{z^{p-k-n}}{p!}\\
&=\binom{n+k}{k} (-1)^{k}\sum_{p=0}^{\infty}(-1)^p \langle p-k\rangle_n\Biggl[\sum_{\ell=0}^{k}(-1)^{\ell} \binom{k}{\ell} \sum_{q=0}^{p}(2q)!!\binom{\ell/2}{q} S(p,q)\Biggr]\\
&\quad\times\lim_{z\to0}\frac{z^{p-k-n}}{p!},
\end{align*}
where we used the generating function~\eqref{2stirl-gen-f}. Accordingly, we obtain
\begin{multline*}
\bell_{n+k,k}\Bigl(\bigl[\sqrt{g(z)}\,\bigr]'\big|_{z=0}, \bigl[\sqrt{g(z)}\,\bigr]''\big|_{z=0}, \dotsc,\bigl[\sqrt{g(z)}\,\bigr]^{(n+1)}\big|_{z=0}\Bigr)\\
=\frac{(-1)^{n}}{k!} \sum_{\ell=0}^{k}(-1)^{\ell} \binom{k}{\ell} \Biggl[\sum_{q=0}^{n+k}(2q)!!\binom{\ell/2}{q} S(n+k,q)\Biggr]
\end{multline*}
for $n\ge k\in\mathbb{N}_0$. Replacing $n$ by $n-k$ reduces to
\begin{multline*}
\bell_{n,k}\Bigl(\bigl[\sqrt{g(z)}\,\bigr]'\big|_{z=0}, \bigl[\sqrt{g(z)}\,\bigr]''\big|_{z=0}, \dotsc, \bigl[\sqrt{g(z)}\,\bigr]^{(n-k+1)}\big|_{z=0}\Bigr)\\
=\frac{(-1)^{n-k}}{k!} \sum_{\ell=0}^{k}(-1)^{\ell} \binom{k}{\ell} \Biggl[\sum_{q=0}^{n}S(n,q)(2q)!!\binom{\ell/2}{q}\Biggr].
\end{multline*}
The series expansion~\eqref{sqrt-exp-wilf-ser-eq} in Theorem~\ref{sqrt-exp-wilf-ser-thm} implies that
\begin{equation*}
\bigl[\sqrt{g(z)}\,\bigr]^{(n)}\big|_{z=0}
=(-1)^n\sum_{k=0}^{n} (-1)^{k-1}S(n,k)(2k-3)!!, \quad n\in\mathbb{N}.
\end{equation*}
Consequently, it follows that
\begin{multline*}
\bell_{n,k}\Biggl(-1, 0, -1, -3,\dotsc, (-1)^{n-k+1}\sum_{\ell=0}^{n-k+1} (-1)^{\ell-1}S(n-k+1,\ell)(2\ell-3)!!\Biggr)\\
=\frac{(-1)^{n-k}}{k!} \sum_{\ell=0}^{n}S(n,\ell)(2\ell)!!\sum_{m=0}^{k}(-1)^{m} \binom{k}{m}\binom{m/2}{\ell}.
\end{multline*}
Employing the identity~\eqref{Bell(n-k)} results in the closed-form formula
\begin{multline}\label{bell-polyn-sqrt-exp-No}
\bell_{n,k}\Biggl(1, 0, 1, 3, 16, 105, \dotsc, (-1)^{n-k}\sum_{\ell=0}^{n-k+1} (-1)^{\ell-1}S(n-k+1,\ell)(2\ell-3)!!\Biggr)\\
=\frac{(-1)^{n}}{k!} \sum_{\ell=0}^{n}S(n,\ell)(2\ell)!!\sum_{m=0}^{k}(-1)^{m} \binom{k}{m}\binom{m/2}{\ell}
\end{multline}
for $n\ge k\in\mathbb{N}_0$. Further considering the formula~\eqref{Stack-Quest-Answ-Eq} in Lemma~\ref{Stack-Quest-Answ-lem}, we can conclude~\eqref{bell-polyn-sqrt-exp} from~\eqref{bell-polyn-sqrt-exp-No}. The proof of Theorem~\ref{sqrt-g(x)=0-bell-thm} is complete.
\end{proof}

\begin{rem}
The closed-form formula~\eqref{bell-polyn-sqrt-exp} in Theorem~\ref{sqrt-g(x)=0-bell-thm} can be applied to establish the Maclaurin power series expansion of functions of the form $F\bigl(\sqrt{2\te^{-z}-1}\,\bigr)$, if all derivatives of the function $F(z)$ at $z=1$ are computable.
\end{rem}

\section{Maclaurin power series expansions of Wilf function}\label{sec-Maclaurin-Wilf-Ser}

In this section, by means of the Taylor power series expansions established in Section~\ref{arctan-taylor-maclaurin-sec} for the last two functions in~\eqref{four-arctangent}, with the help of some preparations in Section~\ref{preparation-sec}, we find two forms for the Maclaurin power series expansion of the Wilf function $W(z)$ defined by~\eqref{Wilf-F-arctan} on the open disc centered at $z=0$ with a radius $\ln2$. We also obtain a closed-form formula for the Gauss hypergeometric function in~\eqref{2F1nnn-z-1}.

\begin{thm}\label{Wilf-F-Ser-Exapn-thm}
For $|z|<\ln2$, the Wilf function $W(z)$ has the Maclaurin power series expansion
\begin{equation}\label{Wilf-F-Ser-Exapn}
W(z)=\sum_{n=0}^{\infty}(-1)^n \Biggl[\sum_{k=0}^{n}(-1)^k S(n,k) \frac{(2k)!!}{2k+1} {}_2F_1\biggl(k+\frac{1}{2},k+1;k+\frac{3}{2};-1\biggr)\Biggr] \frac{z^n}{n!},
\end{equation}
where ${}_2F_1(\alpha,\beta;\gamma;z)$ is the Gauss hypergeometric function defined by~\eqref{Gauss-HF-dfn} and $S(n,k)$ stands for the Stirling numbers of the second kind in~\eqref{2stirl-gen-f}.
\end{thm}

\begin{proof}
Let
\begin{equation*}
H(z)=\frac{\arctan\sqrt{z}\,}{\sqrt{z}\,},\quad \Re(z)>0.
\end{equation*}
Then
\begin{equation*}
H(z)=\int_0^1 \frac{1}{s^2 z+1} \td s \quad\text{and}\quad H^{(n)}(z)=(-1)^nn!\int_0^1 \frac{s^{2n}}{(s^2 z+1)^{n+1}} \td s
\end{equation*}
for $n\in\mathbb{N}_0$. Hence,
\begin{equation*}
H^{(n)}(1)=(-1)^n n! \int_0^1\frac{s^{2n}}{(1+s^2)^{n+1}} \td s.
\end{equation*}
Taking $s=\sqrt{u}\,$ yields $\td s=\frac{\td u}{2\sqrt{u}\,}$ and
\begin{equation*}
\int_0^1 \frac{s^{2n}}{(1+s^2)^{n+1}}\td s =\frac12 \int_0^1 \frac{u^{n-1/2}}{(1+u)^{n+1}}\td u.
\end{equation*}
Letting $z=-1$, $a=n+1$, $b=n+\frac{1}{2}$, and $c=n+\frac{3}{2}$ in~\eqref{Euler-Integral-Gauss-HF} gives
\begin{equation*}
{\,}_2F_1\biggl(n+1,n+\frac{1}{2};n+\frac{3}{2};-1\biggr)
=\biggl(n+\frac{1}{2}\biggr)\int_0^1 \frac{u^{n-1/2}}{(1+u)^{n+1}}\td u.
\end{equation*}
This means that
\begin{equation}\label{arctan-sqrt-lim-n-deriv-n}
H^{(n)}(1)=\frac{(-1)^n n!}{2n+1}{\,}_2F_1\biggl(n+1,n+\frac{1}{2};n+\frac{3}{2};-1\biggr), \quad n\in\mathbb{N}_0.
\end{equation}
\par
The Wilf function $W(z)$ can be regarded as the composite $H(g(z))$, where $w=g(z)=2\te^{-z}-1\to1$ as $z\to0$. Accordingly, utilizing~\eqref{arctan-sqrt-lim-n-deriv-n}, by the Fa\`a di Bruno formula~\eqref{Bruno-Bell-Polynomial}, we derive
\begin{align*}
W^{(n)}(z)&=[H(g(z))]^{(n)}\\
&=\sum_{k=1}^{n}H^{(k)}(w) \bell_{n,k}\bigl(g'(z),g''(z),\dotsc, g^{(n-k+1)}(z)\bigr)\\
&\to\sum_{k=1}^{n}\lim_{w\to1}H^{(k)}(w) \bell_{n,k}\Bigl(\lim_{z\to0}g'(z), \lim_{z\to0}g''(z),\dotsc, \lim_{z\to0}g^{(n-k+1)}(z)\Bigr)\\
&=\sum_{k=1}^{n} \frac{(-1)^k k!}{2k+1}{}_2F_1\biggl(k+\frac{1}{2},k+1;k+\frac{3}{2};-1\biggr) \bell_{n,k}\bigl(-2,2,\dotsc,(-1)^{n-k+1}2\bigr)\\
&=\sum_{k=1}^{n} \frac{(-1)^k k!}{2k+1}{}_2F_1\biggl(k+\frac{1}{2},k+1;k+\frac{3}{2};-1\biggr)(-1)^n2^k \bell_{n,k}(1,1,\dotsc,1)\\
&=(-1)^n\sum_{k=1}^{n} (-1)^kS(n,k)\frac{(2k)!!}{2k+1}{}_2F_1\biggl(k+\frac{1}{2},k+1;k+\frac{3}{2};-1\biggr)
\end{align*}
as $z\to0$ for $n\in\mathbb{N}$, where we used the identities~\eqref{Bell(n-k)} and~\eqref{Bell-stirling}.
Consequently, the series expansion of the Wilf function $W(z)$ is
\begin{multline}\label{Wilf-F-Ser-Prev}
W(z)=\frac{\pi}{4} \\
+\sum_{n=1}^{\infty}(-1)^n\Biggl[\sum_{k=1}^{n} (-1)^kS(n,k)\frac{(2k)!!}{2k+1} {}_2F_1\biggl(k+\frac{1}{2},k+1;k+\frac{3}{2};-1\biggr)\Biggr] \frac{z^n}{n!}.
\end{multline}
\par
In~\cite[p.~109, Example~5.1]{Temme-96-book}, it is given that
\begin{equation*}
{}_2F_1\biggl(\frac{1}{2},1;\frac{3}{2};-z^2\biggr)
=\frac{\arctan z}{z}, \quad |z|<1.
\end{equation*}
By virtue of Abel's limit theorem~\cite[p.~245, Theorem~9.31]{Apostol-MA-1974}, we can take $z=1$ and obtain
\begin{equation*}
{}_2F_1\biggl(\frac{1}{2},1;\frac{3}{2};-1\biggr)
=\frac{\pi}{4}.
\end{equation*}
Combining this with~\eqref{Wilf-F-Ser-Prev} yields the series expansion~\eqref{Wilf-F-Ser-Exapn}. The proof of Theorem~\ref{Wilf-F-Ser-Exapn-thm} is complete.
\end{proof}

\begin{cor}\label{Gauss-HF-Spec-Value-cor}
For $n\in\mathbb{N}_0$, the Gauss hypergeometric function ${}_2F_1(\alpha,\beta;\gamma;z)$ satisfies
\begin{multline}\label{Gauss-HF-Spec-Value}
{}_2F_1\biggl(n+\frac{1}{2},n+1;n+\frac{3}{2};-1\biggr)\\
=\frac{(2n+1)!!}{(2n)!!}\frac{\pi}{4}+\frac{2n+1}{2^{2n}}\sum_{k=1}^{n} (-1)^{k} \binom{2n-k}{n} \frac{2^{k/2}}{k}\sin\frac{3k\pi}{4}.
\end{multline}
\end{cor}

\begin{proof}
From the series expansion~\eqref{arctan-sqrt-expan-eq} in Theorem~\ref{arctan-sqrt-ser-expan-thm}, it follows that
\begin{multline}\label{arctan-square-n-deriv-z=1}
\lim_{z\to1}\biggl(\frac{\arctan\sqrt{z}\,}{\sqrt{z}\,}\biggr)^{(n)}\\
=\frac{(-1)^n}{2^n}\Biggl[\frac{(2n-1)!!\pi}{4}+\frac{n!}{2^{n}}\sum_{k=1}^{n}(-1)^{k} \binom{2n-k}{n} \frac{2^{k/2}}{k}\sin\frac{3k\pi}{4}\Biggr]
\end{multline}
for $n\in\mathbb{N}_0$. Comparing~\eqref{arctan-square-n-deriv-z=1} with~\eqref{arctan-sqrt-lim-n-deriv-n} and simplifying result in the explicit formula~\eqref{Gauss-HF-Spec-Value}. The proof of Corollary~\ref{Gauss-HF-Spec-Value-cor} is complete.
\end{proof}

\begin{rem}
The closed-form formula~\eqref{Gauss-HF-Spec-Value} in Corollary~\ref{Gauss-HF-Spec-Value-cor} can be reformulated as
\begin{multline}\label{Pi-2F1-expression}
\frac{\pi}{4}
=\frac{(2n)!!}{(2n+1)!!}{}_2F_1\biggl(n+\frac{1}{2},n+1;n+\frac{3}{2};-1\biggr)\\
-\frac{(n!)^2}{(2n)!}\sum_{k=1}^{n} (-1)^{k} \binom{2n-k}{n} \frac{2^{k/2}}{k}\sin\frac{3k\pi}{4}
\end{multline}
for all $n\in\mathbb{N}_0$. Since the quantities $2^{k/2}\sin\frac{3k\pi}{4}$ for $k\in\mathbb{N}$ are all integers and ${}_2F_1\bigl(n+\frac{1}{2},n+1;n+\frac{3}{2};-1\bigr)$ for $n\in\mathbb{N}_0$ is an infinite series, we can regard~\eqref{Pi-2F1-expression} as a series representation of the circular constant $\pi$.
\end{rem}

\begin{rem}
In~\cite[Theorem~3]{axioms-2962911.tex}, the author and his collaborator derived
\begin{equation*}
{\,}_2F_1\biggl(n+\frac{1}{2},n+1;n+\frac{3}{2};-z^2\biggr)
=P_n\bigl(z^2\bigr)\frac{\arctan z}{z} +Q_n\bigl(z^2\bigr)\frac{1}{1+z^2}
\end{equation*}
for $n\in\mathbb{N}_0$, where
\begin{equation*}
P_n(z)=\frac{(2n+1)!!}{(2n)!!}\frac{1}{z^n}, \quad n\in\mathbb{N}_0
\end{equation*}
and
\begin{equation*}
Q_n(z)=-\frac{P_n(z)}{(1+z)^{n-1}} \sum_{k=0}^{n-1}\Biggl[\sum_{j=0}^{k}\frac{(-1)^j}{2j+1}\binom{n}{k-j}\Biggr]z^k, \quad n\in\mathbb{N}_0.
\end{equation*}
In particular, in~\cite[Corollary~2]{axioms-2962911.tex}, the author and his collaborator deduced
\begin{equation}\label{2F1-1-explicitO}
{\,}_2F_1\biggl(n+\frac{1}{2},n+1;n+\frac{3}{2};-1\biggr)
=\frac{(2n+1)!!}{(2n)!!}\Biggl[\frac{\pi}{4}-\frac{1}{2^{n}} \sum_{j=0}^{n-1}\frac{(-1)^j}{2j+1} \sum_{\ell=0}^{n-j-1}\binom{n}{\ell}\Biggr]
\end{equation}
for $n\in\mathbb{N}_0$, where an empty sum is understood to be $0$. See also related texts in~\cite{Gauss-Milovanovic-Qi.tex}.
\par
We observe that the right-hand side of~\eqref{2F1-1-explicitO} involves a double summation, whereas the right-hand side of~\eqref{Gauss-HF-Spec-Value} contains a single summation.
\end{rem}

\begin{rem}
Influenced by the paper~\cite{axioms-2962911.tex} and the arXiv preprint~\cite{Qi-Wilf-arXiv-v3}, with the aid of two formulas in~\cite{Pearson-Master-Thesis-2009, Rakha-Rathie-2011}, Sivaraman et~al. established in the note~\cite{Sivaraman-IJAASR-2024-Is2} an alternative form
\begin{equation*}
{\,}_2F_1\biggl(n+\frac{1}{2},n+1;n+\frac{3}{2};-1\biggr)
=\frac{\pi}{2}\frac{(2n+1)!!}{(2n)!!} -\frac{2n+1}{2^{n+2}}\sum_{r=0}^{n}\binom{n}{r}\frac{2^r}{r!} \biggl[\Gamma\biggl(\frac{r+1}{2}\biggr)\biggr]^2
\end{equation*}
for $n\in\mathbb{N}_0$. This form is different from~\eqref{2F1-1-explicitO} and~\eqref{Gauss-HF-Spec-Value}.
\end{rem}

\begin{thm}\label{Wilf-F-Ser-final-thm}
For $|z|<\ln2$, the Wilf function $W(z)$ defined by~\eqref{Wilf-F-arctan} has the Maclaurin power series expansion
\begin{multline}\label{Wilf-F-Ser-final}
W(z)=\sum_{n=0}^{\infty}(-1)^n \Biggl[\frac{\pi}{4}\sum_{k=0}^{n}(-1)^k S(n,k)(2k-1)!!\\*
+\sum_{k=1}^{n}(-1)^k S(n,k) \frac{k!}{2^k}\sum_{\ell=1}^{k} (-1)^{\ell} \binom{2k-\ell}{k} \frac{2^{\ell/2}}{\ell} \sin\frac{3\ell\pi}{4}\Biggr] \frac{z^n}{n!},
\end{multline}
where $S(n,k)$ denotes the Stirling numbers of the second kind in~\eqref{2stirl-gen-f}.
\end{thm}

\begin{proof}
The series expansion~\eqref{Wilf-F-Ser-final} follows from substituting the closed-form formula~\eqref{Gauss-HF-Spec-Value} in Corollary~\ref{Gauss-HF-Spec-Value-cor} into the Maclaurin power series expansion~\eqref{Wilf-F-Ser-Exapn} in Theorem~\ref{Wilf-F-Ser-Exapn-thm} and simplifying.
The proof of Theorem~\ref{Wilf-F-Ser-final-thm} is complete.
\end{proof}

\begin{rem}
In~\cite[Section~4]{axioms-2962911.tex}, among other things, the author and his collaborator presented
\begin{multline}\label{Wilf-F-Ser-final-rewr}
W(z)=\frac{\pi}{4}+\sum_{n=1}^{\infty}(-1)^n \Biggl[\sum_{k=1}^{n}(-1)^k S(n,k)(2k-1)!!\\
\times\Biggl(\frac{\pi}{4}-\frac{1}{2^k}\sum_{j=0}^{k-1}\frac{(-1)^j}{2j+1} \sum_{\ell=0}^{k-j-1}\binom{k}{\ell}\Biggr)\Biggr] \frac{z^n}{n!}
\end{multline}
for $|z|<\ln2$.
\par
We observe that the Maclaurin power series expansion in~\eqref{Wilf-F-Ser-final} involves a triple summation, whereas the expansion in~\eqref{Wilf-F-Ser-final-rewr} involves a quadruple summation.
\end{rem}

\section{Series expansions of inverse hyperbolic tangent function}\label{ser-arctanh-sec}

By means of the relation
\begin{equation*}
\arctan z=-\ti\arctanh(\ti z), \quad z^2\ne-1
\end{equation*}
in~\cite[p.~80, 4.4.22]{abram}, from the series expansions in~\eqref{arctan-pi-4-ser-eq}, \eqref{arctan-frac-ser-mid}, \eqref{arctan-sqrt-expan-eq}, and~\eqref{Wilf-F-Ser-final}, we derive the following series expansions related to the inverse hyperbolic tangent function $\arctanh z$.
\begin{enumerate}
\item
For $|z-\ti|<\sqrt{2}\,$, we have
\begin{equation*}
\frac{\arctanh z}{z}=\sum_{n=0}^{\infty}\biggl[\frac{\pi}{4}+T(n)\biggr](\ti z+1)^n,
\end{equation*}
where $T(n)$ is defined by~\eqref{T(n)-dfn-notation}.
\item
For $|z-\ti|<1$, we have
\begin{equation*}
\frac{\arctanh z-\frac{\pi}{4}\ti}{z}=\sum_{n=1}^{\infty}T(n)(\ti z+1)^n.
\end{equation*}
\item
For $|z+1|<2$, we have
\begin{multline*}
\frac{\arctanh\sqrt{z}\,}{\sqrt{z}\,}\\
=\sum_{n=0}^{\infty} \frac{1}{2^n}\Biggl[\frac{(2n-1)!!\pi}{4} +\frac{n!}{2^{n}}\sum_{k=1}^{n} (-1)^{k} \binom{2n-k}{n} \frac{2^{k/2}}{k} \sin\frac{3k\pi}{4}\Biggr] \frac{(z+1)^n}{n!}.
\end{multline*}
\item
For $|z|<\ln2$, we have
\begin{multline*}
\frac{\arctanh\sqrt{1-2\te^{-z}}\,}{\sqrt{1-2\te^{-z}}\,}
=\sum_{n=0}^{\infty}(-1)^n \Biggl[\frac{\pi}{4}\sum_{k=0}^{n}(-1)^k S(n,k)(2k-1)!!\\
+\sum_{k=1}^{n}(-1)^k S(n,k) \frac{k!}{2^k} \sum_{\ell=1}^{k} (-1)^{\ell} \binom{2k-\ell}{k} \frac{2^{\ell/2}}{\ell}\sin\frac{3\ell\pi}{4}\Biggr] \frac{z^n}{n!},
\end{multline*}
where $S(n,k)$ denotes the Stirling numbers of the second kind in~\eqref{2stirl-gen-f}.
\end{enumerate}

\section{Analysis on coefficients in series expansions of Wilf function}\label{coefficients-analysis-sec}

Considering the power series expansion~\eqref{Wilf-F-Ser-final} in Theorem~\ref{Wilf-F-Ser-final-thm}, we can see that the sequences $b_n$ and $c_n$ in~\eqref{Wilf-F-Ser-Split2Part} can be expressed as
\begin{equation}\label{b(n)-closed-rational}
b_n=\frac{1}{4}\frac{(-1)^n}{n!} \sum_{k=0}^{n}(-1)^kS(n,k) (2k-1)!!
\end{equation}
and
\begin{equation}\label{c(n)-closed-rational}
c_n=\frac{(-1)^{n+1}}{n!}\sum_{k=1}^{n}(-1)^k S(n,k) \frac{k!}{2^k}\sum_{\ell=1}^{k} (-1)^{\ell} \binom{2k-\ell}{k} \frac{2^{\ell/2}}{\ell}\sin\frac{3\ell\pi}{4}
\end{equation}
for $n\in\mathbb{N}_0$, where $S(n,k)$ denotes the Stirling numbers of the second kind.
\par
It is easy to see that the explicit formulas in~\eqref{b(n)-closed-rational} and~\eqref{c(n)-closed-rational} are closed-form expressions and that the sequence $b_n$ is rational. Since $2^{\ell/2}\sin\frac{3\ell\pi}{4}$ for $\ell\in\mathbb{N}$ is an integer, the sequence $c_n$ is also rational.
\par
From closed-form formulas in~\eqref{b(n)-closed-rational} and~\eqref{c(n)-closed-rational}, we can easily compute and list the first twelve values of the sequences $b_n$ and $c_n$ in Table~\ref{first12values-b-c-n}.
\begin{table}[hbtp]
\caption{The first twelve values of $b_n$ and $c_n$}
\begin{tabular}{|c||c|c|c|c|c|c|c|c|c|c|c|c|}
\hline
$n$&0&1&2&3&4&5&6&7&8&9&10&11\\ \hline
$b_n$ & $\frac{1}{4}$ & $\frac{1}{4}$ & $\frac{1}{4}$ & $\frac{7}{24}$ & $\frac{35}{96}$ & $\frac{113}{240}$ & $\frac{1787}{2880}$ & $\frac{16717}{20160}$ & $\frac{2257}{2016}$ & $\frac{315883}{207360}$ & $\frac{4324721}{2073600}$ & $\frac{447448}{155925}$ \\ \hline
$c_n$ & $0$ & $\frac{1}{2}$ & $\frac{3}{4}$ & $\frac{11}{12}$ & $\frac{55}{48}$ & $\frac{71}{48}$ & $\frac{2807}{1440}$ & $\frac{8753}{3360}$ & $\frac{94541}{26880}$ & $\frac{694663}{145152}$ & $\frac{47552791}{7257600}$ & $\frac{719718067}{79833600}$\\ \hline
\end{tabular}
\label{first12values-b-c-n}
\end{table}
These twenty-four special values coincide with the corresponding ones listed in the table on~\cite[p.~592]{Ward-2010-Wilf}.

\begin{thm}\label{b(n)-c(n)-seq-GF-thm}
For $n\in\mathbb{N}_0$, the sequences $b_n$ and $c_n$ can be generated respectively by
\begin{equation}\label{b(n)-GF}
G(x)=\frac{1}{4\sqrt{2\te^{-x}-1}\,}=\sum_{n=0}^{\infty}b_nx^n, \quad |x|<\ln2
\end{equation}
and
\begin{equation}\label{c(n)-GF}
\frac{\frac{\pi}{4}-\arctan\sqrt{2\te^{-x}-1}\,}{\sqrt{2\te^{-x}-1}\,}=\sum_{n=0}^{\infty}c_nx^n, \quad |x|<\ln2.
\end{equation}
\end{thm}

\begin{proof}
In light of the Fa\`a di Bruno formula~\eqref{Bruno-Bell-Polynomial} and the identities~\eqref{Bell(n-k)} and~\eqref{Bell-stirling}, we obtain
\begin{align*}
4[G(x)]^{(n)}&=\sum_{k=0}^{n} \bigl(v^{-1/2}\bigr)^{(k)} \bell_{n,k}\bigl(-2\te^{-x}, 2\te^{-x}, \dotsc, (-1)^{n-k+1}2\te^{-x}\bigr)\\
&=\sum_{k=0}^{n} \biggl\langle-\frac{1}{2}\biggr\rangle_k v^{-1/2-k} (-1)^n2^k \te^{-kx} \bell_{n,k}(1, 1, \dotsc, 1)\\
&=(-1)^n\sum_{k=0}^{n} (-1)^k(2k-1)!! (2\te^{-x}-1)^{-1/2-k} \te^{-kx} S(n,k)\\
&\to(-1)^n\sum_{k=0}^{n}(-1)^k S(n,k)(2k-1)!!
\end{align*}
as $x\to0$ for $n\in\mathbb{N}_0$, where $v=v(x)=2\te^{-x}-1$. This derivative implies that, by comparing with the explicit expression in~\eqref{b(n)-closed-rational} of $b_n$, the function $G(x)$ defined on the interval $(-\infty,\ln2)$ is a generating function of the sequence $b_n$ for $n\in\mathbb{N}_0$.
\par
From the series expansions~\eqref{Wilf-F-Ser-Split2Part} and~\eqref{b(n)-GF}, we derive
\begin{equation*}
\sum_{n=0}^{\infty}c_nx^n=\pi G(x)-W(x)=\frac{\frac{\pi}{4}-\arctan\sqrt{2\te^{-x}-1}\,}{\sqrt{2\te^{-x}-1}\,},
\end{equation*}
the series expansion~\eqref{c(n)-GF} is thus proved. The proof of Theorem~\ref{b(n)-c(n)-seq-GF-thm} is complete.
\end{proof}

\begin{rem}
We can regard the generating function in~\eqref{c(n)-GF} as a composite of the functions $\frac{\frac{\pi}{4}-\arctan z}{z}$ and $z=z(x)=\sqrt{g(x)}\,=\sqrt{2\te^{-x}-1}\,$. Therefore, with the aid of the Fa\`a di Bruno formula~\eqref{Bruno-Bell-Polynomial}, we obtain
\begin{align*}
&\quad\biggl(\frac{\frac{\pi}{4}-\arctan\sqrt{2\te^{-x}-1}\,}{\sqrt{2\te^{-x}-1}\,}\biggr)^{(n)}\\
&=\sum_{k=0}^{n}\biggl(\frac{\frac{\pi}{4}-\arctan z}{z}\biggr)^{(k)} \bell_{n,k}\bigl(z'(x), z''(x), \dotsc, z^{(n-k+1)}(x)\bigr)\\
&\to\sum_{k=0}^{n}(-1)^{k+1}k!T(k) \bell_{n,k}\Biggl(-1,0,-1,-3,-16,-105,\dotsc,\\
&\quad(-1)^{n-k+1}\sum_{\ell=0}^{n-k+1} (-1)^{\ell-1}S(n-k+1,\ell)(2\ell-3)!!\Biggr)\\
&=-\sum_{k=0}^{n}k!T(k) \bell_{n,k}\Biggl(1,0,1,3,16,105,\dotsc,\\
&\quad(-1)^{n-k}\sum_{\ell=0}^{n-k+1} (-1)^{\ell-1}S(n-k+1,\ell)(2\ell-3)!!\Biggr)\\
&=(-1)^{n+1}\sum_{k=0}^{n}T(k)\sum_{\ell=k}^{n}(-1)^{\ell}S(n,\ell)P(\ell,k)\\
&=(-1)^{n+1}\sum_{\ell=0}^{n}(-1)^{\ell}S(n,\ell)\frac{\ell!}{2^{\ell}}\sum_{k=1}^{\ell}(-1)^{k}\binom{2\ell-k}{\ell}\frac{2^{k/2}}{k} \sin\frac{3k\pi}{4}
\end{align*}
as $x\to0$ and $z\to1$, where we used the explicit formula~\eqref{bell-polyn-sqrt-exp} in~Theorem~\ref{sqrt-g(x)=0-bell-thm} and the identity~\eqref{Curious-id-inv-TT}. The closed-form expression~\eqref{c(n)-closed-rational} of the sequence $c_n$ is thus recovered.
\end{rem}

\begin{rem}
The generating function $G(x)$ for $|x|<\ln2$ in~\eqref{b(n)-GF} can be rearranged as
\begin{align*}
G(x)&=\frac{\te^{x/2}}{4\sqrt{2}\,\sqrt{1-\te^x/2}\,}\\
&=\frac{1}{4\sqrt{2}\,}\sum_{j=0}^{\infty}\binom{-1/2}{j}\biggl(-\frac{1}2\biggr)^j\te^{(j+1/2)x}\\
&=\frac{1}{4\sqrt{2}\,}\sum_{j=0}^{\infty}\binom{-1/2}{j}\biggl(-\frac{1}2\biggr)^j \sum_{n=0}^{\infty}\frac{1}{n!}\biggl(j+\frac{1}{2}\biggr)^nx^n\\
&=\frac{1}{4\sqrt{2}\,} \sum_{n=0}^{\infty}\frac{1}{n!} \Biggl[\sum_{j=0}^{\infty}\binom{-1/2}{j}\biggl(-\frac{1}2\biggr)^j \biggl(j+\frac{1}{2}\biggr)^n\Biggr]x^n\\
&=\frac{1}{4\pi}\sum_{n=0}^{\infty}\Biggl[\sum_{j=0}^{\infty} \biggl(j+\frac{1}{2}\biggr)^{n-1}\frac{2^{j+1/2}}{\binom{2j+1}{j+1/2}}\Biggr] \frac{x^n}{n!}.
\end{align*}
Then the infinite series representation
\begin{equation}\label{b(n)-series-express}
b_n=\frac{1}{4\pi}\frac{1}{n!}\sum_{j=0}^{\infty} \biggl(j+\frac{1}{2}\biggr)^{n-1}\frac{2^{j+1/2}}{\binom{2j+1}{j+1/2}}, \quad n\in\mathbb{N}_0,
\end{equation}
which appears at the website \url{https://oeis.org/A014307}, is recovered.
\end{rem}

\begin{thm}\label{b(n)-posit-ineq-thm}
The following conclusions are valid:
\begin{enumerate}
\item
the function $G(x)$ is logarithmically absolutely monotonic on $(-\infty,\ln2)$;
\item
the sequences
\begin{equation}\label{n!b(n)}
4n!b_n=(-1)^n\sum_{k=0}^{n} (-1)^k S(n,k)(2k-1)!!, \quad n\in\mathbb{N}_0
\end{equation}
and
\begin{equation}\label{(2n-2)-log-conv}
(-1)^n\sum_{k=1}^{n}(-1)^{k}S(n,k)(2k-2)!!, \quad n\in\mathbb{N}
\end{equation}
are positive, nondecreasing, and logarithmically convex.
\end{enumerate}
\end{thm}

\begin{proof}
For $n\in\mathbb{N}$ and $v=v(x)=2\te^{x}-1$, making use of the Fa\`a di Bruno formula~\eqref{Bruno-Bell-Polynomial} and the identities~\eqref{Bell(n-k)} and~\eqref{Bell-stirling}, we obtain
\begin{align*}
[\ln G(-x)]^{(n)}&=\biggl(\ln\frac{1}{4\sqrt{2\te^{x}-1}\,}\biggr)^{(n)}\\
&=-\frac12\sum_{k=1}^{n}(\ln v)^{(k)} \bell_{n,k}\bigl(2\te^{x},2\te^{x},\dotsc,2\te^{x}\bigr)\\
&=-\frac12\sum_{k=1}^{n}\frac{(-1)^{k-1}(k-1)!}{v^{k}} 2^k\te^{kx} \bell_{n,k}(1,1,\dotsc,1)\\
&=\sum_{k=1}^{n}(-1)^k\frac{(2k-2)!!}{(2\te^{x}-1)^{k}} \te^{kx} S(n,k)\\
&\to\sum_{k=1}^{n}(-1)^{k} S(n,k) (2k-2)!!, \quad x\to0
\end{align*}
and, by the same arguments,
\begin{align*}
[\ln G(x)]^{(n)}&=(-1)^n\sum_{k=1}^{n}(-1)^k\frac{(2k-2)!!}{(2\te^{-x}-1)^{k}} \te^{-kx} S(n,k)\\
&\to(-1)^n\sum_{k=1}^{n}(-1)^{k} S(n,k) (2k-2)!!, \quad x\to0.
\end{align*}
\par
On the other hand, it is obvious that the first logarithmic derivative
\begin{align*}
[\ln G(x)]'&=\biggl(\ln\frac{1}{4\sqrt{2\te^{-x}-1}\,}\biggr)'
=\frac{1}{2-\te^x}
=\frac{1}{2}\sum_{k=0}^{\infty}\biggl(\frac{\te^{x}}{2}\biggr)^k, \quad x<\ln2
\end{align*}
is absolutely monotonic on $(-\infty,\ln2)$. Hence, the function $-[\ln G(-x)]'$ is completely monotonic on $(-\ln2,\infty)$, and therefore $G(-x)$ is logarithmically completely monotonic on $(-\ln2,\infty)$. As a result, the positivity of the sequence in~\eqref{(2n-2)-log-conv} is valid.
\par
Based on~\cite[Definition~1 and Theorem~1]{absolute-mon-simp.tex} recited in Section~\ref{preparation-sec}, considering the logarithmically absolute monotonicity of $G(x)$, we derive that the function $G(x)$ is absolutely monotonic on $(-\infty,\ln2)$. Combining this result with the proof of Theorem~\ref{b(n)-c(n)-seq-GF-thm} leads to the positivity of the sequence in~\eqref{n!b(n)}.
\par
Applying the positivity of the sequence in~\eqref{n!b(n)} to the closed-form formula in~\eqref{b(n)-closed-rational} for $b_n$ reveals that the sequence $b_n$ for $n\in\mathbb{N}_0$ is positive.
\par
In~\cite[p.~369]{mpf-1993}, it was stated that, if $f$ is a completely monotonic function on $(0,\infty)$ such that $f^{(n)}(t)\ne0$ for $n\in\mathbb{N}_0$, then the sequence $(-1)^{n}f^{(n)}(t)$ for $n\in\mathbb{N}_0$ is logarithmically convex. Applying this result to the functions $G(-x)$ and $-[\ln G(-x)]'$, both of which are completely monotonic on $[0,\infty)\subset(-\ln2,\infty)$, yields the logarithmic convexity of the sequences in~\eqref{n!b(n)} and~\eqref{(2n-2)-log-conv}.
\par
From the logarithmic convexity of the sequences in~\eqref{n!b(n)} and~\eqref{(2n-2)-log-conv} and by the increasing property of the first few values of the sequences in~\eqref{n!b(n)} and~\eqref{(2n-2)-log-conv}, we conclude that both of the sequences in~\eqref{n!b(n)} and~\eqref{(2n-2)-log-conv} are nondecreasing.
The proof of Theorem~\ref{b(n)-posit-ineq-thm} is complete.
\end{proof}

\begin{rem}
A few results of the integer sequence $4n!b_n$ for $n\in\mathbb{N}_0$ have been collected at the website \url{https://oeis.org/A014307}. For example, the recursive relation
\begin{equation*}
(n+1)!b_{n+1}=\frac14+\sum_{j=1}^n\biggl[\binom{n+1}{j}-1\biggr]j!b_j, \quad n\in\mathbb{N}_0
\end{equation*}
is listed at the website, which implies the positivity and nondecreasing property of the sequence $n!b_n$ for $n\in\mathbb{N}_0$.
\par
From the complete monotonicity of the functions $G(-x)$ and $-[\ln G(-x)]'$ on the infinite interval $(0,\infty)\subset(-\ln2,\infty)$ and a number of properties of completely monotonic functions collected in~\cite[Chapter~XIII]{mpf-1993},~\cite[Chapter~IV]{Widder-1941B}, and the monograph~\cite{Schilling-Song-Vondracek-2nd}, we can deduce many properties, which do not appear at the website \url{https://oeis.org/A014307}, of the sequence $4n!b_n$ for $n\in\mathbb{N}_0$, as done in~\cite{1st-Sirling-Number-2012-Ren.tex, MMAS-19-15811.tex} and closely related references therein. Due to the limitation of length, we do not give the details here.
\end{rem}

\begin{rem}
Many results for the sequence $2n!c_n$ with $n\in\mathbb{N}$ have been listed at \url{http://oeis.org/A180875}.
At the website, we do not see the generating function in~\eqref{c(n)-GF} for the sequence $c_n$ with $n\in\mathbb{N}_0$.
\end{rem}

\begin{thm}\label{c(n)-b(n)-ratio-lim-thm}
The limits
\begin{equation}\label{b(n)-lim-infty}
\lim_{n\to\infty}b_n=\infty
\end{equation}
and
\begin{equation}\label{c(n)-b(n)-ratio-lim-eq}
\lim_{n\to\infty}\frac{c_n}{b_n}=\pi
\end{equation}
are valid. Consequently, we have
\begin{equation}\label{c(n)-lim-infty}
\lim_{n\to\infty}c_n=\infty.
\end{equation}
\end{thm}

\begin{proof}
In~\cite[p.~2, (1.2)]{Wallis-Qi-Cao-Niu.tex}, the double inequality
\begin{equation}\label{best-bounds-Wallis}
\frac{1}{\sqrt{\pi(n+4/\pi-1)}\,}\le \frac{(2n-1)!!}{(2n)!!}<\frac{1}{\sqrt{\pi(n+1/4)}\,}, \quad n\in\mathbb{N}
\end{equation}
was recited, where the constants $\frac{4}{\pi}-1$ and $\frac14$ in \eqref{best-bounds-Wallis} are the best possible in the sense that these two constants cannot be replaced by smaller and bigger ones, respectively.
In~\cite[Theorem~1.1]{Wallis-Qi-Cao-Niu.tex}, the double inequality
\begin{equation}\label{Wallis-type-ineq}
\frac{\sqrt{\pi}\,}{2\sqrt{n+{9\pi}/{16}-1}\,}
\le\frac{(2n)!!}{(2n+1)!!}
<\frac{\sqrt{\pi}\,}{2\sqrt{n+3/4}\,},\quad n\in\mathbb{N}
\end{equation}
was established, where the constants $\frac{9\pi}{16}-1$ and $\frac34$ in \eqref{Wallis-type-ineq} are the best possible in the sense that these two numbers cannot be replaced by smaller and bigger ones, respectively; see also~\cite[p.~658, Theorem~4]{notes-best-simple-open-jkms.tex}.
\par
The infinite series representation~\eqref{b(n)-series-express} can be reformulated as
\begin{equation}\label{b(n)-series-express0reform}
b_n=\frac{1}{n!}\sum_{j=0}^{\infty} \biggl(j+\frac{1}{2}\biggr)^{n-1}\frac{(2j+1)!!}{(2j)!!} \frac{1}{2^{j+7/2}}, \quad n\in\mathbb{N}_0.
\end{equation}
Substituting~\eqref{Wallis-type-ineq} into~\eqref{b(n)-series-express0reform} yields
\begin{align*}
b_n&>\frac{1}{n!}\Biggl[\frac{1}{2^{n+5/2}} +\frac{1}{2^{5/2}\sqrt{\pi}\,} \sum_{j=1}^{\infty}\biggl(j+\frac{1}{2}\biggr)^{n-1} \frac{\sqrt{j+3/4}\,}{2^{j}}\Biggr]\\
&>\frac{1}{n!}\Biggl[\frac{1}{2^{n+5/2}} +\frac{1}{2^{5/2}\sqrt{\pi}\,} \sum_{j=1}^{\infty}\biggl(j+\frac{1}{2}\biggr)^{n-1/2} \frac{1}{2^{j}}\Biggr]\\
&=\frac{1}{n!}\biggl[\frac{1}{2^{n+5/2}} +\frac{1}{2^{7/2}\sqrt{\pi}\,} \Phi\biggl(\frac{1}{2},\frac{1}{2}-n,\frac{3}{2}\biggr)\biggr]
\end{align*}
for $n\in\mathbb{N}_0$, where $\Phi(z,s,a)$ is the Lerch transcendent
\begin{equation*}
\Phi(z,s,a)=\sum_{j=0}^{\infty}\frac{z^{j}}{(j+a)^{s}}, \quad |z|<1, \quad a\ne0,-1,-2,\dotsc;
\end{equation*}
see~\cite[p.~612, Section~25.14]{NIST-HB-2010}.
\par
Let $0<|\lambda|<1$ and $\lambda\not\in(-1,0)$. Corollary~4 in~\cite[p.~29]{Navas-JAT-2013} reads that the limit
\begin{equation*}
\lim_{\mathbb{R}\ni s\to-\infty}\frac{(-\ln\lambda)^{1-s}\Phi(\lambda,s,z)}{\Gamma(1-s)}=\frac{1}{\lambda^z}
\end{equation*}
converges uniformly for $z$ on compact subsets of $\mathbb{C}$. As a result, we arrive at
\begin{align*}
\lim_{n\to\infty}\biggl[\frac{1}{n!}\Phi\biggl(\frac{1}{2},\frac{1}{2}-n,\frac{3}{2}\biggr)\biggr]
&=\lim_{n\to\infty}\biggl[\frac{\Gamma(1/2+n)}{n!(\ln2)^{1/2+n}} \frac{(\ln2)^{1/2+n}}{\Gamma(1/2+n)} \Phi\biggl(\frac{1}{2},\frac{1}{2}-n,\frac{3}{2}\biggr)\biggr]\\
&=\frac{2^{3/2}\sqrt{\pi}\,}{(\ln2)^{1/2}}\lim_{n\to\infty}\frac{(2n-1)!!}{(2n)!!(\ln2)^{n}}\\
&=\infty,
\end{align*}
where we used in the last step the double inequality~\eqref{best-bounds-Wallis}. Consequently, the limit~\eqref{b(n)-lim-infty} is proved.
\par
By Item~\ref{Ward-item3} on Page~\pageref{Ward-item3}, \(a_n=\frac{h(n)}{R^n}\), where \(R>1\) and \(h(n)\) is subexponential. Hence, we obtain \(a_n\to0\). Since \(a_n=b_n\bigl(\pi-\frac{c_n}{b_n}\bigr)\) and \(b_n\to\infty\), it follows that
$$ \pi-\frac{c_n}{b_n}=\frac{a_n}{b_n}\to0. 
$$
Therefore \(\frac{c_n}{b_n}\to\pi\) as $n\to\infty$.
The limit~\eqref{c(n)-b(n)-ratio-lim-eq} is thus verified.
\par
From both the limits~\eqref{b(n)-lim-infty} and~\eqref{c(n)-b(n)-ratio-lim-eq}, we can deduce directly the limit~\eqref{c(n)-lim-infty}.
The proof of Theorem~\ref{c(n)-b(n)-ratio-lim-thm} is complete.
\end{proof}

\begin{rem}
Since $b_n$ and $c_n$ for all $n\in\mathbb{N}_0$ are rational numbers, we can regard the limit~\eqref{c(n)-b(n)-ratio-lim-eq} as an asymptotic rational approximation to the circular constant $\pi$.
\end{rem}

\begin{rem}
At the websites \url{https://oeis.org/A014307} and \url{http://oeis.org/A180875}, there exist two asymptotic formulas
\begin{equation}\label{b(n)-asymptotic-eq}
4n!b_n\sim\frac{\sqrt{2}\,n^n}{\te^n(\ln2)^{n+1/2}}
=\sqrt{\frac{2}{\ln2}}\,\biggl(\frac{n}{\te\ln2}\biggr)^n, \quad n\to\infty
\end{equation}
and
\begin{equation}\label{c(n)-asymptotic-eq}
2(n+1)!c_{n+1}\sim\pi \frac{n^{n+1}}{\sqrt{2}\te^{n}(\ln2)^{n+3/2}}
=\frac{\te\pi}{\sqrt{2\ln2}\,}\biggl(\frac{n}{\te\ln2}\biggr)^{n+1}, \quad n\to\infty,
\end{equation}
respectively.
On 4 September 2021, V\'aclav Kot\v{e}\v{s}ovec told the author that, using the Maple module equivalent, one can prove these two asymptotic formulas~\eqref{b(n)-asymptotic-eq} and~\eqref{c(n)-asymptotic-eq}.
Accordingly, it is easy to see that
\begin{equation*}
\frac{c_n}{b_n}\sim\te\pi\biggl(1-\frac{1}{n}\biggr)^n\to\pi, \quad n\to\infty.
\end{equation*}
The limit~\eqref{c(n)-b(n)-ratio-lim-eq} is thus proved again.
\end{rem}

\begin{rem}
We conjecture that $b_n$ is nondecreasing and convex for $n\in\mathbb{N}_0$ (strictly increasing for $n\ge2$), and that $c_n$ is increasing for $n\in\mathbb{N}_0$ and convex for $n\ge2$.
\end{rem}

\begin{thm}\label{d(n)>0-thm}
Let
\begin{equation}\label{d(n)-defined-eq}
d_n=\frac{(-1)^{n}}{n!} \sum_{k=0}^{n} (-1)^{k}S(n,k)(2k-3)!!, \quad n\in\mathbb{N}_0.
\end{equation}
For $n\in\mathbb{N}$, the sequence $d_n$ is nonnegative.
\end{thm}

\begin{proof}
Differentiating on both sides of the series expansion~\eqref{sqrt-exp-wilf-ser-eq} in Theorems~\ref{sqrt-exp-wilf-ser-thm} and simplifying give
\begin{equation}\label{A014304-series-expan-eq}
\frac{1}{\sqrt{\te^x(2-\te^x)}\,}=\sum_{n=0}^{\infty}(n+1)d_{n+1}x^n, \quad |x|<\ln2.
\end{equation}
Taking the logarithm and differentiating result in
\begin{equation*}
\biggl[\ln\frac{1}{\sqrt{\te^x(2-\te^x)}\,}\biggr]'
=\frac{1}{2-\te^x}-1
=\frac{1}{2}\Biggl[\sum_{k=1}^{\infty}\biggl(\frac{\te^x}{2}\biggr)^k-1\Biggr],
\end{equation*}
which is absolutely monotonic on the interval $[0,\ln2)$. This means that the function $\frac{1}{\sqrt{\te^x(2-\te^x)}\,}$ is (logarithmically) absolutely monotonic on $[0,\ln2)$. Since the logarithmically absolute monotonicity of $F$ implies that $F$ is absolutely monotonic, and hence $F^{(n)}(0)\ge 0$ for all $n\in\mathbb{N}_0$, we arrive at
\begin{equation*}
\lim_{x\to0^+}\biggl[\frac{1}{\sqrt{\te^x(2-\te^x)}\,}\biggr]^{(n)}=(n+1)!d_{n+1}\ge0, \quad n\in\mathbb{N}_0.
\end{equation*}
The proof of Theorem~\ref{d(n)>0-thm} is complete.
\end{proof}

\begin{rem}
We can rewrite the generating function in~\eqref{A014304-series-expan-eq} as
\begin{align*}
\frac{1}{\sqrt{\te^x(2-\te^x)}\,}&=\frac{\te^{-x/2}}{\sqrt{2}\,\sqrt{1-\te^x/2}\,}\\
&=\frac{1}{\sqrt{2}\,}\sum_{j=0}^{\infty}\binom{-1/2}{j}\biggl(-\frac{1}2\biggr)^j\te^{(j-1/2)x}\\
&=\frac{1}{\sqrt{2}\,}\sum_{j=0}^{\infty}\binom{-1/2}{j}\biggl(-\frac{1}2\biggr)^j \sum_{n=0}^{\infty}\frac{1}{n!}\biggl(j-\frac{1}{2}\biggr)^nx^n\\
&=\frac{1}{\sqrt{2}\,} \sum_{n=0}^{\infty}\frac{1}{n!} \Biggl[\sum_{j=0}^{\infty}\binom{-1/2}{j}\biggl(-\frac{1}2\biggr)^j \biggl(j-\frac{1}{2}\biggr)^n\Biggr]x^n\\
&=\frac{1}{\pi}\sum_{n=0}^{\infty}\Biggl[\sum_{j=0}^{\infty} \frac{(j-1/2)^{n}}{j+1/2} \frac{2^{j+1/2}}{\binom{2j+1}{j+1/2}}\Biggr] \frac{x^n}{n!}.
\end{align*}
Therefore, we obtain an infinite series representation
\begin{equation}\label{d(n)-series-repres}
d_{n}=\frac{1}{\pi} \frac{1}{n!} \sum_{j=0}^{\infty} \frac{(j-1/2)^{n-1}}{j+1/2} \frac{2^{j+1/2}}{\binom{2j+1}{j+1/2}}, \quad n\in\mathbb{N}.
\end{equation}
\par
By an approach similar to that used in the proof of the limit~\eqref{b(n)-lim-infty}, as in the proof of Theorem~\ref{c(n)-b(n)-ratio-lim-thm}, we can confirm that $\lim_{n\to\infty}d_n=\infty$.
\end{rem}

\begin{rem}
The sequence $n!d_n$ appears to be an integer sequence. The first few values of $n!d_n$ for $0\le n\le21$ are
\begin{gather*}
-1,\quad 1,\quad 0,\quad 1,\quad 3,\quad 16,\quad 105,\quad 841,\quad 7938,\quad 86311,\quad 1062435,\\
14605306,\quad 221790723,\quad 3687263581,\quad 66609892440,\quad 1299237505021,\\
27213601303983,\quad 609223983928576,\quad 14516520372130245,\\
366820998284761861,\quad 9798039716677045218,\quad 275837214061454446171.
\end{gather*}
When $n\in\mathbb{N}$, that is, when omitting the first term $-1$, the sequence $n!d_n$ appears at the website \url{https://oeis.org/A014304}.
\par
We conjecture that $n!d_n$ is increasing for $n\ge2$ and convex for $n\ge1$, and that $d_n$ is increasing for $n\ge4$ and convex for $n\ge3$.
\end{rem}

\begin{thm}\label{b(n)-det-thm}
For $n\in\mathbb{N}$, the sequences $b_n$ and $d_n$ are connected by the relations
\begin{equation}\label{b(n)-d(n)-Det}
b_n=\frac{1}{4}
\begin{vmatrix}
1 & -1 & 0 & 0 & \dotsm & 0& 0 & 0\\
0 & 1 & -1 & 0 & \dotsm & 0 & 0& 0\\
\frac16 & 0 & 1 & -1 & \dotsm & 0 & 0 & 0\\
\vdots & \vdots & \vdots & \vdots & \ddots & \vdots &\vdots & \vdots\\
d_{n-2} & d_{n-3} & d_{n-4} & d_{n-5} & \dotsm & 1 & -1 & 0\\
d_{n-1} & d_{n-2} & d_{n-3} & d_{n-4} & \dotsm & 0 & 1 & -1\\
d_{n} & d_{n-1} & d_{n-2} & d_{n-3} & \dotsm & \frac16 & 0 & 1
\end{vmatrix}_{n\times n}
\end{equation}
and
\begin{equation}\label{d(n)-b(n)-Det}
d_n=(-1)^{n-1}4^n
\begin{vmatrix}
\frac14 & \frac14 & 0 & 0 & \dotsm & 0& 0 & 0\\
\frac14 & \frac14 & \frac14 & 0 & \dotsm & 0 & 0& 0\\
\frac7{24} & \frac14 & \frac14 & \frac14 & \dotsm & 0 & 0 & 0\\
\vdots & \vdots & \vdots & \vdots & \ddots & \vdots &\vdots & \vdots\\
b_{n-2} & b_{n-3} & b_{n-4} & b_{n-5} & \dotsm & \frac14 & \frac14 & 0\\
b_{n-1} & b_{n-2} & b_{n-3} & b_{n-4} & \dotsm & \frac14 & \frac14 & \frac14\\
b_{n} & b_{n-1} & b_{n-2} & b_{n-3} & \dotsm & \frac7{24} & \frac14 & \frac14
\end{vmatrix}_{n\times n}.
\end{equation}
\end{thm}

\begin{proof}
The Wronski theorem~\cite[p.~17, Theorem~1.3]{Henrici-B-1974} reads that, if $\alpha_0\ne0$ and
\begin{equation*}
P(x)=\alpha_0+\alpha_1x+\alpha_2x^2+\dotsm
\end{equation*}
is a formal series, then the coefficients of the reciprocal series
\begin{equation*}
\frac{1}{P(x)}=\beta_0+\beta_1x+\beta_2x^2+\dotsm
\end{equation*}
are given by
\begin{equation}\label{f(t)g(t)1determ}
\beta_n=\frac{(-1)^n}{\alpha_0^{n+1}}
\begin{vmatrix}
\alpha_1 & \alpha_0 & 0 & 0 & \dotsm & 0& 0 & 0\\
\alpha_2 & \alpha_1 & \alpha_0 & 0 & \dotsm & 0 & 0& 0\\
\alpha_3 & \alpha_2 & \alpha_1 & \alpha_0 & \dotsm & 0 & 0 & 0\\
\vdots & \vdots & \vdots & \vdots & \ddots & \vdots &\vdots & \vdots\\
\alpha_{n-2} & \alpha_{n-3} & \alpha_{n-4} & \alpha_{n-5} & \dotsm & \alpha_1 & \alpha_0 & 0\\
\alpha_{n-1} & \alpha_{n-2} & \alpha_{n-3} & \alpha_{n-4} & \dotsm & \alpha_2 & \alpha_1 & \alpha_0\\
\alpha_{n} & \alpha_{n-1} & \alpha_{n-2} & \alpha_{n-3} & \dotsm & \alpha_3 & \alpha_2 & \alpha_1
\end{vmatrix}, \quad n\in\mathbb{N}.
\end{equation}
The formula~\eqref{f(t)g(t)1determ} can also be found in the literature~\cite[p.~347]{Inselberg-JMAA-1978}, \cite[Lemma~2.1]{CDM-73022.tex}, \cite[Lemma~2.4]{2Closed-Bern-Polyn2.tex}, \cite[Lemma~2.3]{Schroder-Seq-3rd.tex}, and~\cite[Section~2]{Rutishauser-ZAMP-1956}.
\par
Using the Wronski theorem~\eqref{f(t)g(t)1determ}, from the series expansions~\eqref{sqrt-exp-wilf-ser-eq} and~\eqref{b(n)-GF} in Theorems~\ref{sqrt-exp-wilf-ser-thm} and~\ref{b(n)-c(n)-seq-GF-thm}, respectively, we arrive at the relations~\eqref{b(n)-d(n)-Det} and~\eqref{d(n)-b(n)-Det} immediately. The proof of Theorem~\ref{b(n)-det-thm} is thus complete.
\end{proof}

\begin{thm}\label{relation-thm-(n+1)d(n+1)}
For $n\in\mathbb{N}_0$, let
\begin{equation*}
e_n=-\frac{1}{(2n)!!}\sum_{k=0}^{n}\binom{n}{k}2^k \sum_{\ell=0}^{k} S(k,\ell) \frac{(2\ell-3)!!}{2^\ell}.
\end{equation*}
Then the Maclaurin power series expansion
\begin{equation}\label{recip-A014304-series-expan-eq}
\sqrt{\te^x(2-\te^x)}\,=\sum_{n=0}^{\infty}e_n x^n, \quad |x|<\ln2
\end{equation}
holds and the relations
\begin{equation}\label{relation-1-(n+1)d(n+1)}
d_{n+1}=\frac{(-1)^n}{n+1}
\begin{vmatrix}
0 & 1 & 0 & 0 & \dotsm & 0& 0 & 0\\
-\frac{1}{2} & 0 & 1 & 0 & \dotsm & 0 & 0& 0\\
-\frac{1}{2} & -\frac{1}{2} & 0 & 1 & \dotsm & 0 & 0 & 0\\
\vdots & \vdots & \vdots & \vdots & \ddots & \vdots &\vdots & \vdots\\
e_{n-2} & e_{n-3} & e_{n-4} & e_{n-5} & \dotsm & 0 & 1 & 0\\
e_{n-1} & e_{n-2} & e_{n-3} & e_{n-4} & \dotsm & -\frac{1}{2} & 0 & 1\\
e_{n} & e_{n-1} & e_{n-2} & e_{n-3} & \dotsm & -\frac{1}{2} & -\frac{1}{2} & 0
\end{vmatrix}_{n\times n}
\end{equation}
and
\begin{equation}\label{relation-2-(n+1)d(n+1)}
e_n=(-1)^n
\begin{vmatrix}
0 & 1 & 0 & \dotsm & 0& 0 & 0\\
\frac12 & 0 & 1 & \dotsm & 0 & 0& 0\\
\frac12 & \frac12 & 0 & \dotsm & 0 & 0 & 0\\
\vdots & \vdots & \vdots & \ddots & \vdots &\vdots & \vdots\\
(n-1)d_{n-1} & (n-2)d_{n-2} & (n-3)d_{n-3} & \dotsm & 0 & 1 & 0\\
nd_{n} & (n-1)d_{n-1} & (n-2)d_{n-2} & \dotsm & \frac12 & 0 & 1\\
(n+1)d_{n+1} & nd_{n} & (n-1)d_{n-1} & \dotsm & \frac12 & \frac12 & 0
\end{vmatrix}_{n\times n}
\end{equation}
are valid for $n\in\mathbb{N}$, where $d_n$ for $n\in\mathbb{N}_0$ are defined by~\eqref{d(n)-defined-eq}.
\end{thm}

\begin{proof}
It is standard that
\begin{align*}
\bigl[\sqrt{\te^x(2-\te^x)}\,\bigr]^{(n)}
&=\bigl(\te^{x/2}\sqrt{2-\te^x}\,\bigr)^{(n)}\\
&=\sum_{k=0}^{n}\binom{n}{k}\bigl(\sqrt{2-\te^x}\,\bigr)^{(k)}\bigl(\te^{x/2}\bigr)^{(n-k)}\\
&\hskip-2em=\sum_{k=0}^{n}\binom{n}{k}\frac{\te^{x/2}}{2^{n-k}} \sum_{\ell=0}^{k} \biggl\langle\frac{1}{2}\biggr\rangle_{\ell} \bigl(2-\te^x\bigr)^{1/2-\ell} \bell_{k,\ell}\bigl(-\te^x,-\te^x,\dotsc,-\te^x\bigr)\\
&=-\frac{1}{2^n}\sum_{k=0}^{n}\binom{n}{k}2^k\te^{x/2} \sum_{\ell=0}^{k} \frac{(2\ell-3)!!}{2^\ell} \bigl(2-\te^x\bigr)^{1/2-\ell} \te^{\ell x}S(k,\ell)\\
&\to -\frac{1}{2^n}\sum_{k=0}^{n}\binom{n}{k}2^k \sum_{\ell=0}^{k} \frac{(2\ell-3)!!}{2^\ell} S(k,\ell)
\end{align*}
as $x\to0$, where we used the Fa\`a di Bruno formula~\eqref{Bruno-Bell-Polynomial} and two identities~\eqref{Bell(n-k)} and~\eqref{Bell-stirling}.
\par
Combining the series expansions~\eqref{A014304-series-expan-eq} and~\eqref{recip-A014304-series-expan-eq} with the Wronski theorem in~\eqref{f(t)g(t)1determ} leads to two relations~\eqref{relation-1-(n+1)d(n+1)} and~\eqref{relation-2-(n+1)d(n+1)} for $n\in\mathbb{N}$. The proof of Theorem~\ref{relation-thm-(n+1)d(n+1)} is complete.
\end{proof}

\begin{rem}
The generating function in~\eqref{recip-A014304-series-expan-eq} can be alternatively expanded as
\begin{align*}
\sqrt{\te^x(2-\te^x)}\,&=\sqrt{2}\,\te^{x/2}\sqrt{1-\frac{\te^x}{2}}\,\\
&=\sqrt{2}\,\sum_{j=0}^{\infty}\binom{1/2}{j}\biggl(-\frac{1}{2}\biggr)^j\te^{(j+1/2)x}\\
&=\sqrt{2}\, \sum_{j=0}^{\infty} \binom{1/2}{j} \biggl(-\frac{1}{2}\biggr)^j \sum_{n=0}^{\infty}\biggl(j+\frac{1}{2}\biggr)^n\frac{x^n}{n!}\\
&=\sqrt{2}\, \sum_{n=0}^{\infty}\Biggl[\sum_{j=0}^{\infty} \binom{1/2}{j} \biggl(-\frac{1}{2}\biggr)^j \biggl(j+\frac{1}{2}\biggr)^n\Biggr]\frac{x^n}{n!}\\
&=-\frac{1}{\pi} \sum_{n=0}^{\infty}\Biggl[\sum_{j=0}^{\infty} \frac{(j+1/2)^{n-1}}{j-1/2} \frac{2^{j+1/2}}{\binom{2j+1}{j+1/2}}\Biggr]\frac{x^n}{n!}.
\end{align*}
Accordingly, we obtain an infinite series representation
\begin{equation}\label{e(n)-series-repres}
e_n=-\frac{1}{\pi}\frac{1}{n!} \sum_{j=0}^{\infty} \frac{(j+1/2)^{n-1}}{j-1/2} \frac{2^{j+1/2}}{\binom{2j+1}{j+1/2}}, \quad n\in\mathbb{N}_0.
\end{equation}
\par
By an approach similar to that used in the proof of the limit~\eqref{b(n)-lim-infty}, as in the proof of Theorem~\ref{c(n)-b(n)-ratio-lim-thm}, we can confirm that $\lim_{n\to\infty}e_n=-\infty$.
\end{rem}

\begin{rem}
The sequence $n!e_n$ for $n\in\mathbb{N}_0$ is an integer sequence. Its first few values for $0\le n\le20$ are
\begin{gather*}
1,\quad 0,\quad -1,\quad -3,\quad -10,\quad -45,\quad -271,\quad -2058,\quad -18775,\\
-199335,\quad -2410516,\quad -32683563,\quad -490870315,\\
-8087188200,\quad -144994236661,\quad -2810079139143,\\
-58536519252130,\quad -1304198088413265,\quad -30946462816602331,\\
-779104979758256298,\quad -20742005411397108595.
\end{gather*}
\end{rem}

\begin{rem}
The Wilf function $W(z)$ in~\eqref{Wilf-F-arctan} can be generalized as
\begin{equation}\label{Wilf-F-arctan-Gen}
W_{\alpha,\beta,\rho}(z)=\frac{\arctan\sqrt{\alpha\te^{\beta z}+\rho}\,}{\sqrt{\alpha\te^{\beta z}+\rho}\,},
\end{equation}
where $\alpha,\beta,\rho\in\mathbb{R}$ are constants such that the function $W_{\alpha,\beta,\rho}(z)$ is defined on a region in the \(z\)-plane.
\par
It is obvious that $W_{2,-1,-1}(z)=W(z)$.
\par
It would be interesting to study the function $W_{\alpha,\beta,\rho}(z)$ defined in~\eqref{Wilf-F-arctan-Gen}.
\end{rem}

\section{Conclusions}

In this paper, with the aid of the Fa\`a di Bruno formula~\eqref{Bruno-Bell-Polynomial}, by virtue of several identities from~\eqref{Bell(n-k)} to~\eqref{113-final-formula} for the partial Bell polynomials, with the help of two combinatorial identities~\eqref{sum-central-binom-ell-eq} and~\eqref{Stack-Quest-Answ-Eq}, and by means of the Wronski theorem~\eqref{f(t)g(t)1determ}, and in light of the (logarithmically) absolute monotonicity of the generating function $G(x)$ in~\eqref{b(n)-GF}, we established the Taylor power series expansions~\eqref{arctan-pi-4-ser-eq}, \eqref{arctan-frac-ser-mid}, \eqref{arctan-sqrt-expan-eq}, and those in Section~\ref{ser-arctanh-sec}, obtained the Maclaurin power series expansions~\eqref{Wilf-F-Ser-Exapn} and~\eqref{Wilf-F-Ser-final} for the Wilf function~\eqref{Wilf-F-arctan}, and analyzed several properties of the coefficients \(b_n\) and \(c_n\), including their generating functions and limiting behavior, as well as positivity, monotonicity, and logarithmic convexity of two associated sequences. These coefficients are closed-form expressions in~\eqref{b(n)-closed-rational} and~\eqref{c(n)-closed-rational}.
We also derived a closed-form formula~\eqref{Gauss-HF-Spec-Value} for a sequence of special values of the Gauss hypergeometric function in~\eqref{2F1nnn-z-1}, discovered a closed-form formula~\eqref{bell-polyn-sqrt-exp} in Theorem~\ref{sqrt-g(x)=0-bell-thm} for the partial Bell polynomials in~\eqref{Bell-Poly-Stirling-invol}, presented several infinite series representations~\eqref{d(n)-series-repres} and~\eqref{e(n)-series-repres}, recovered an asymptotic rational approximation in Theorem~\ref{c(n)-b(n)-ratio-lim-thm} to the circular constant $\pi$, and connected several integer sequences by determinants in Theorems~\ref{b(n)-det-thm} and~\ref{relation-thm-(n+1)d(n+1)}.
\par
This paper is a revised version of the arXiv preprint \url{https://doi.org/10.48550/arXiv.2110.08576}. Its first version, arXiv:2110.08576v1 (\url{https://arxiv.org/abs/2110.08576v1}), was cited in the three papers~\cite{Prakash-Kumar-2023, axioms-2962911.tex, Gauss-Milovanovic-Qi.tex}.

\section{Declarations}

\begin{description}
\item[Acknowledgements]
The author thanks
\begin{enumerate}
\item
Darij Grinberg (Drexel University, USA) and several other persons for their discussions and suggestions of the descending induction used in the proof of Lemma~\ref{sum-central-binom-lem-ell} at the website \url{https://mathoverflow.net/q/402822/} (accessed on 30 August 2021);
\item
V\'aclav Kot\v{e}\v{s}ovec (Czech Republic, retired, \url{http://www.kotesovec.cz/}) for providing two figures plotted with Maple about the asymptotic formulas of the sequences $b_n$ and $c_n$ via several e-mails on 3--4 September 2021;
\item
Hjalmar Rosengren (Chalmers University of Technology and University of Gothenburg, Sweden) and several other persons for their discussions and an alternative proof of the infinite series representation~\eqref{b(n)-series-express} at the website \url{https://mathoverflow.net/q/403281/} (accessed on 6 September 2021);
\item
the anonymous referees for their careful corrections, helpful suggestions, and valuable comments on the original version of this paper.
\end{enumerate}

\item[Funding]
The author was supported in part by the Natural Science Foundation of Inner Mongolia Autonomous Region (Grant No.~2025QN01041) and by the Youth Project of Hulunbuir City for Basic Research and Applied Basic Research (Grant No.~GH2024020).

\item[Authors' Contributions]
Not applicable. 

\item[Competing Interests]
The author declares that he has no competing interests.

\item[Availability of Data and Material]
Data sharing is not applicable to this article as no new data were created or analyzed in this study.

\item[Institutional Review Board Statement]
Not applicable.

\item[Informed Consent Statement]
Not applicable.

\item[Ethical Approval]
The conducted research is not related to either human or animal use.

\item[Use of AI Tools Declaration]
The author declares he has not used Artificial Intelligence (AI) tools in the creation of this article.

\end{description}

\bibliographystyle{mmn}
\bibliography{Wilf-Ward-2010P}

\end{document}